\documentclass[11pt,a4 paper,oneside]{amsart}
\usepackage[utf8]{inputenc}
\usepackage[OT1]{fontenc}
\usepackage[english]{babel}
\usepackage{dsfont}
\usepackage{amsfonts}
\usepackage{amsmath}
\usepackage{bbm}
\usepackage{wasysym}
\usepackage{amsthm}
\usepackage{amssymb}
\usepackage{float}
\usepackage{textcomp}
\usepackage{xcolor}
\usepackage{graphicx}
\usepackage{fancybox}
\usepackage{fancyhdr}
\usepackage{mathrsfs}
\usepackage{tikz}
\usepackage{centernot}
\usepackage{listings}
\usepackage{faktor}
\usepackage{bm}
\usepackage{setspace}
\usepackage{hyperref}
\usepackage{newlfont}
\usepackage{geometry}
\usepackage{ccaption}
\usepackage{tipa}
\usepackage[autostyle,italian=guillemets]{csquotes}

\usepackage{guit}
\usepackage{tikz-cd}
\usepackage{enumerate}

\theoremstyle{definition}
\newtheorem{Def}{Definition}[section]
\newtheorem{es}[Def]{Example}
\newtheorem{ese}[Def]{Examples}

\theoremstyle{remark}
\newtheorem{obs}[Def]{Remark}
\theoremstyle{plain}
\newtheorem{prop}[Def]{Proposition}
\newtheorem{lema}[Def]{Lemma}
\newtheorem{cor}[Def]{Corollary}
\newtheorem{teo}[Def]{Theorem}

\newcommand{\fqv}{finitary quasivariety}
\newcommand{\lfp}{locally finitely presentable}
\newcommand{\bo}{\mathbf}
\newcommand{\A}{{\mathcal A}}
\newcommand{\B}{{\mathcal B}}
\newcommand{\C}{{\mathcal C}}
\newcommand{\D}{{\mathcal D}}
\newcommand{\E}{{\mathcal E}}

\newcommand{\G}{{\mathcal G}}

\newcommand{\K}{{\mathcal K}}
\renewcommand{\L}{{\mathcal L}}
\newcommand{\M}{{\mathcal M}}
\newcommand{\N}{{\mathcal N}}

\renewcommand{\P}{{\mathcal P}}

\newcommand{\R}{{\mathcal R}}

\newcommand{\V}{{\mathcal V}}

\title{Enriched Regular Theories}
\author{Stephen Lack and Giacomo Tendas}
\address{Department of Mathematics and Statistics, Macquarie University NSW 2109, 
	Australia}
\email{steve.lack@mq.edu.au}
\address{Department of Mathematics and Statistics, Macquarie University NSW 2109, 
	Australia}
\email{giacomo.tendas@mq.edu.au}
\date{\today}
\thanks{The first-named author acknowledges with gratitude the support of the Australian Research Council Grant DP190102432. The second-named author gratefully acknowledges the support of an International Macquarie University Research Excellence Scholarship.}

\begin{document}
	
\begin{abstract}
	Regular and exact categories were first introduced by Michael Barr in 1971; since then, the theory has developed and found many applications in algebra, geometry, and logic. In particular, a small regular category determines a certain theory, in the sense of logic, whose models are the regular functors into Set. Barr further showed that each small and regular category can be embedded in a particular category of presheaves; then in 1990 Makkai gave a simple explicit characterization of the essential image of the embedding, in the case where the original regular category is moreover exact. More recently Prest and Rajani, in the additive context, and Kuber and Rosick\'y, in the ordinary one, described a duality which connects an exact category with its (definable) category of models. Working over a suitable base for enrichment, we define an enriched notion of regularity and exactness, and prove a corresponding version of the theorems of Barr, of Makkai, and of Prest-Rajani/Kuber-Rosick\'y.
\end{abstract}	
	
\maketitle
	
\tableofcontents

\section{Introduction}

When talking about {\em theories} we may think of two different approaches: a logical one and a categorical one. From the logical point of view, a theory is given by a list of axioms on a fixed set of operations, and its models are corresponding sets and functions that satisfy those axioms. For instance {\em algebraic theories} are those whose axioms consist of equations based on the operation symbols of the language (e.g. the axioms for abelian groups or rings).
More generally, if the axioms are still equations but the operation symbols are not defined globally, but only on equationally defined subsets, we talk of {\em essentially algebraic theories}.

\begin{es}
	Sets with a binary relation can be seen as the models of the essentially algebraic theory with two global operations $s,t$ $:\textnormal{edge}\to \textnormal{vertex}$ (source and target), a partial operation $\sigma$ $:\textnormal{edge}\times \textnormal{edge}\to \textnormal{edge}$ such that $\sigma(x,y)$ is defined if and only if $s(x)=s(y)$ and $t(x)=t(y)$.
	The axioms of the theory are then: $\sigma(x,y)=x$, $\sigma(x,y)=y$.
\end{es}

A further step can be made by considering {\em regular theories}, in which we allow existential quantification over the usual equations.

\begin{es}
	Von Neumann regular rings are the models of the regular theory with axioms those of rings plus the following one:
	$\forall x\ \exists y\ x=xyx$.
\end{es}

Categorically speaking, we could think of a theory as a category $\C$ with some structure, and of a model of $\C$ as a functor $F:\C\to\bo{Set}$ which preserves that structure; this approach was first introduced by Lawvere in \cite{Law63:articolo}. Algebraic theories then correspond to categories with finite products, and models are finite product preserving functors. On the other hand, a category with finite limits represents an essentially algebraic theory, and functors preserving finite limits are its models \cite{Fre72:articolo}. Regular theories \cite{MR77:libro} correspond instead to {\em regular categories}: finitely complete ones with coequalizers of kernel pairs, for which regular epimorphisms are pullback stable. Models here are functors preserving finite limits and regular epimorphisms; we refer to them as regular functors.

These two notions, categorical and logical, can be recovered from each other: given a logical theory, there is a syntactic way to build a category with the relevant structure for which models of the theory correspond to functors to $\bo{Set}$ preserving this structure, and vice versa. 

For essentially algebraic theories there is a duality between theories and their models:

\begin{teo}[Gabriel-Ulmer, \cite{GU71:libro}]
	The following is a biequivalence of 2-categories:
	\begin{center}
		
		\begin{tikzpicture}[baseline=(current  bounding  box.south), scale=2]

		\node (f) at (0,0.4) {$\textnormal{LFP}(-,\bo{Set}):\bo{LFP}$};
		\node (g) at (2.2,0.4) {$\bo{Lex}^{op}:\textnormal{Lex}(-,\bo{Set}),$};
		
		\path[font=\scriptsize]

		([yshift=1.3pt]f.east) edge [->] node [above] {} ([yshift=1.3pt]g.west)
		([yshift=-1.3pt]f.east) edge [<-] node [below] {} ([yshift=-1.3pt]g.west);
		\end{tikzpicture}

	\end{center}
	where $\bo{LFP}$ is the 2-category of \lfp\ categories, finitary right adjoints, and natural transformations; and $\bo{Lex}$ is the 2-category of small categories with finite limits, finite-limit-preserving functors, and natural transformations.
\end{teo}

There is a corresponding duality in the context of regular theories; to describe it let us recall the most important results involving regular categories. First of all, Barr proved in \cite{Bar86:articolo} that every small regular category can be regularly embedded in the functor category based on its models:

\begin{teo}[Barr's Embedding]
	Let $\C$ be a small regular category; then the evaluation functor $\textnormal{ev}:\C\to[\textnormal{Reg}(\C,\bo{Set}),\bo{Set}]$ is fully faithful and regular.
\end{teo}

Later Makkai proved in \cite{Mak90:articolo} that if the category $\C$ is moreover exact in the sense of Barr \cite{Bar86:articolo} (also called effective regular in \cite{Joh02:libro}), then it can be recovered from its category of models $\textnormal{Reg}(\C,\bo{Set})$ as follows:

\begin{teo}[Makkai's Image Theorem]
	Let $\C$ be a small exact category. The essential image of the embedding $\textnormal{ev}:\C\to[\textnormal{Reg}(\C,\bo{Set}),\bo{Set}]$ is given by those functors which preserve filtered colimits and small products.
\end{teo}

On one side of the duality there is the 2-category $\bo{Ex}$ of small exact categories, regular functors, and natural transformations. On the other side is a 2-category whose objects are called {\em definable categories}, and which will be categories of models of some regular theory. A category is definable if it is a full subcategory of a \lfp\ category closed under small products, filtered colimits, and pure subobjects; equivalently it is a finite injectivity class in a \lfp\ category. This is a less explicit definition than that of \lfp\ categories, in that it refers to an ``external'' \lfp\ category in which the definable category embeds. A morphism between definable categories is then a functor that preserves filtered colimits and products; denote by $\bo{DEF}$ the corresponding 2-category. The duality can hence be expressed as:

\begin{teo}
	The following is a biequivalence of 2-categories:
	\begin{center}
		
		\begin{tikzpicture}[baseline=(current  bounding  box.south), scale=2]

		\node (f) at (0,0.4) {$\textnormal{DEF}(-,\bo{Set}):\bo{DEF}$};
		\node (g) at (2.3,0.4) {$\bo{Ex}^{op}:\textnormal{Reg}(-,\bo{Set})$};
		
		\path[font=\scriptsize]

		([yshift=1.3pt]f.east) edge [->] node [above] {} ([yshift=1.3pt]g.west)
		([yshift=-1.3pt]f.east) edge [<-] node [below] {} ([yshift=-1.3pt]g.west);
		\end{tikzpicture}
		
	\end{center}
\end{teo}

This was proved in the additive setting in \cite[Theorem~2.3]{PR10:articolo}, where it becomes a biequivalence between the 2-category of additive definable categories and the opposite of the 2-category of small abelian categories. The version appearing above was formulated as \cite[Theorem~3.2.5]{KR18:articolo}, but the proof presented there is incomplete, as we explain in Section \ref{equiv}.

Gabriel-Ulmer duality has been extended to the enriched context by Kelly in \cite{Kel82:articolo}. Our aim is to extend the other three theorems, finding a common setting that includes both the ordinary and the additive context. Note that an enriched version of Barr's Embedding Theorem has already been considered in \cite{Chi11:articolo}, but the notion of regularity appearing there is more restrictive than ours: see Remark \ref{Chi}. 

First we need to specify our assumptions on the base for enrichment we are going to work with. Start as usual \cite{Kel82:libro} with a symmetric monoidal closed complete and cocomplete category $\V=(\V_0,I,\otimes)$; since we want to talk about finite weighted limits and regularity, this should at least be \lfp\ as a closed category (in the sense of \cite{Kel82:articolo}) and regular. In fact we ask something more, our bases for enrichment will generally be (unsorted) {\em finitary varieties}: categories of the form $\text{FP}(\C,\bo{Set})$, consisting of finite product preserving functors for some small category $\C$ with finite products. Equivalently a finitary variety can be described as an exact and cocomplete category with a strong generator made of finitely presentable (regular) projective objects. In addition to this, we ask these finitely presentable projective objects to behave well with respect to the monoidal structure (in a sense made clear in Section \ref{fqv}). We call a finitary variety with such a structure a {\em symmetric monoidal finitary variety}; we also consider a generalization, called a {\em symmetric monoidal finitary quasivariety}.

In this context we define an enriched version of regularity and exactness (Section \ref{regcat}) similar to the ordinary ones but with the additional requirement that regular epimorphisms should be stable under finite projective powers. This allows us to prove an enriched version of Barr's Embedding Theorem (Theorem \ref{Barr}), saying that for each small and regular $\V$-category $\C$ the evaluation functor $$\textnormal{ev}:\C\to[\textnormal{Reg}(\C,\V),\V]$$ is a fully faithful regular embedding. If the underlying ordinary category on $\C$ is moreover exact, the essential image of $\textnormal{ev}_{\C}$ is given by those functors that preserve filtered colimits, products, and projective powers (Theorem \ref{Makkai}), recovering an enriched version of Makkai's Image Theorem. We obtain these results for enrichment over a symmetric monoidal \fqv.
 
An enriched notion of definable $\V$-category is also introduced (Section \ref{def}). Then, if our $\V$ is a symmetric monoidal finitary variety, we are able to recover the duality between the 2-category $\V\text{-}\bo{Ex}$ of small exact $\V$-categories, and $\V\text{-}\bo{DEF}$ of definable $\V$-categories (Theorem \ref{main}), showing that each definable $\V$-category is {\em exactly definable}, namely of the form $\text{Reg}(\B,\V)$ for an exact $\V$-category $\B$. In Section \ref{free} we use this to give an explicit description of the free exact completions over finitely complete $\V$-categories and over regular $\V$-categories.

\section{Background Notions}

In this section we recall the main features about enriched categories that we are going to use throughout this paper; the main references for this are \cite{Kel82:libro} and \cite{Kel82:articolo}.

Fix a complete and cocomplete symmetric monoidal closed category $\V=(\V_0,I,\otimes)$. We denote by $[-,-]:\V_0\times\V_0\to\V_0$ the internal hom that makes $\V$ closed, so that $-\otimes Y$ is left adjoint to $[Y,-]$ for each $Y\in\V_0$.

Given a $\V$-category $\C$, which hence has hom-objects $\C(X,Y)$ in $\V_0$, we denote by $\C_0$ the underlying ordinary category of $\C$; this has the same objects as $\C$, but $\C_0(X,Y)=\V_0(I,\C(X,Y))$. Similarly, for any $\V$-functor $F:\C\to\B$ we denote by $F_0:\C_0\to \B_0$ the induced ordinary functor between $\C_0$ and $\B_0$. Note that we allow all our $\V$-categories to be large, unless specified otherwise.

We assume the reader to be familiar with the notion of enriched conical limit; we recall the definitions of power and copower to fix notation. Let $\C$ be a $\V$-category, $C$ an object of $\C$, and $X$ an object of $\V$. The {\em power} in $\C$ of $C$ by $X$, if it exists, is given by an object $C^X$ of $\C$ together with a map $X\to\C(C^X,C)$ inducing a $\V$-natural isomorphism $$\C(B,C^X)\cong[X,\C(B,C)]$$ in $\V_0$. Dual is the notion of {\em copower} of $C$ by $X$, which is denoted by $X\cdot C$.

Given an ordinary locally small category $\K$, denote by $\K_{\V}$ the free $\V$-category over $\K$. Then, for any $\V$-category $\C$ and any ordinary functor $T:\K\to\C_0$, we denote by $\text{lim}T$, if it exists, the conical limit in $\C$ of the corresponding $\V$-functor $T_{\V}:\K_{\V}\to\C$. Such a limit will always give a limit of $T:\K\to\C_0$; conversely, $\text{lim}T$ exists when the ordinary limit of $T:\K\to\C_0$ exists and is preserved by each representable $\C(C,-)_0:\C_0\to\V_0$. This latter preservation condition is automatic if $\C$ has copowers by all objects in a strong generator for $\V_0$, but not in general.

Starting from Section \ref{regcat} our base for enrichment will be, among other things, \lfp\ as a closed category in the sense of Kelly \cite{Kel82:articolo}. In this context we can talk about enriched \lfp\ categories. Given a $\V$-category $\L$; an object $A$ of $\L$ will be called {\em finitely presentable} if the functor $\L(A,-):\L\to\V$ preserves conical filtered colimits; we denote by $\L_f$ the full subcategory of finitely presentable objects. Then a $\V$-category $\L$ will be called \lfp\ as a $\V$-category if it is $\V$-cocomplete and has a small strong generator $\G\subseteq \L_f$.

The notion of {\em finite weighted limit} is also introduced in \cite{Kel82:articolo} giving finiteness conditions on the weights. Then it can be proven that, for a $\V$-category $\C$, having finite weighted limits is the same as having finite conical limits and finite powers, where by the latter we mean powers by finitely presentable objects of $\V_0$. Denote by $\V\text{-}\bo{Lex}$ the 2-category of finitely complete $\V$-categories (namely $\V$-categories with finite weighted limits), finite limit preserving $\V$-functors, and $\V$-natural transformations. Similarly, let $\V\text{-}\bo{LFP}$ be the 2-category of \lfp\ $\V$-categories, right adjoint $\V$-functors that preserve filtered colimits, and $\V$-natural transformations; then we obtain the enriched version of Gabriel-Ulmer duality:

\begin{teo}[Kelly, \cite{Kel82:articolo}]\label{kelly}
	The following is a biequivalence of 2-categories:
	\begin{center}
		
		\begin{tikzpicture}[baseline=(current  bounding  box.south), scale=2]

		\node (f) at (0,0.4) {$(-)_f^{op}:\V\text{-}\bo{LFP}$};
		\node (g) at (2.2,0.4) {$\V\text{-}\bo{Lex}^{op}:\textnormal{Lex}(-,\V)$};
		
		\path[font=\scriptsize]

		([yshift=1.3pt]f.east) edge [->] node [above] {} ([yshift=1.3pt]g.west)
		([yshift=-1.3pt]f.east) edge [<-] node [below] {} ([yshift=-1.3pt]g.west);
		\end{tikzpicture}
		
	\end{center}
\end{teo}

\section{Weak Reflections}

Recall the following definitions for ordinary $\bo{Set}$-enriched categories:

\begin{Def}\label{ofic}
	Given an arrow $h:A\to B$ in a category $\L$, an object $L\in\L$ is said to be $h${\em -injective} if $\L(h,L):\L(B,L)\to\L(A,L)$ is a surjection of sets. Given a small set $\M$ of arrows in $\L$ write $\M$-inj for the full subcategory of $\L$ consisting of those objects which are $h$-injective for each $h\in\M$. Categories arising in this way are called {\em injectivity classes}, or {\em small injectivity classes} for emphasis.
\end{Def}

\begin{Def}
	Let $\D$ be a full subcategory of $\L$ and $L\in\L$; we say that $p:L\to S$ is a {\em weak reflection} of $L$ into $\D$ if $S\in\D$ and each $K\in\D$ is $p$-injective. We say that $\D$ is weakly reflective in $\L$ if each object of $\L$ has a weak reflection into $\D$.
\end{Def}

The first result of this section is a well-known one which relates injectivity classes and weakly reflective subcategories:

\begin{teo}\label{wr}
	Each injectivity class $\M\textnormal{-inj}$ in a locally finitely presentable category $\L$ is weakly reflective. If the domain and codomain of each morphism in $\M$ are finitely presentable, there is a finitary endofunctor $R$ of $\L$ and a natural transformation $r:1_{\L}\to R$, such that each component $r_L:L\to RL$ is a weak reflection of $L$ into $\M\textnormal{-inj}$, and is a transfinite composite of pushouts of morphisms in $\M$.
	
\end{teo}

This follows by Quillen's ``small object argument'' --- see for example \cite[Theorem~2.1.14]{Hov99:articolo} --- the weak reflection of an object $X$ is given by the induced factorization of the unique map $X\to 1$. The fact that $R$ is finitary is perhaps less well-known, but it follows from the fact that the hom-functors $\L(X,-)$ are finitary whenever $X$ is the domain or codomain of a morphism in $\M$, and that $R$ can be constructed as an iterated colimit of these functors: since colimits commute with colimits, any colimit of finitary functors is finitary.

Now we move to the enriched context and consider a corresponding notion of weak reflection. For this, let us fix a symmetric monoidal closed complete and cocomplete category $\V=(\V_0,\otimes,I)$ as our base.

\begin{Def}\label{weakref}
	Let $\L$ be a $\V$-category and $\D$ a full subcategory of $\L$. Given $L\in \L$, a {\em weak reflection} of $L$ into $\D$ is a morphism $p:L\to S$ such that $S\in\D$ and $$\L(p,T):\L(S,T)\to\L(L,T)$$ is a regular epimorphism in $\V$ for each $T\in\D$.
	We say that $\D\subseteq \L$ is weakly reflective if each $L$ in $\L$ has a weak reflection into $\D$. 
\end{Def}

A more refined notion of enriched weak reflection has been considered in \cite{LR12:articolo}, where $\L(p,T)$ was required to come from a specified class $\E$ of morphisms in $\V$. The above definition corresponds to taking $\E$ to consist of the regular epimorphisms.

\begin{prop}\label{codense}
	Let $\L$ be a $\V$-category with coequalizers of kernel pairs and $J:\D\hookrightarrow\L$ be the inclusion of a full weakly reflective subcategory of $\L$ for which the weak reflections can be chosen to be regular monomorphisms. Then $\D$ is codense in $\L$; meaning that the functor 
	\begin{center}
		
		\begin{tikzpicture}[baseline=(current  bounding  box.south), scale=2]
		
		\node (e) at (-0.5,0.4) {$\L(1,J):$};
		\node (f) at (0.1,0.4) {$\L^{op}$};
		\node (g) at (1.3,0.4) {$[\D,\V]$};
		\node (h) at (0.1,0.1) {$\ \  F\ \ \ $};
		\node (i) at (1.3,0.1) {$\ \L(F,J-)  \ $};
		
		\path[font=\scriptsize]

		(f) edge [->] node [above] {} (g)
		(h) edge [dashed,|->] node [above] {} (i);
		\end{tikzpicture}
	\end{center}
	is full and faithful.
\end{prop}
\begin{proof}
	For each $L\in\L$ consider a weak reflection $s:L\rightarrowtail S$, with $S\in\D$, which by hypothesis we can choose to be a regular monomorphism. Take then the cokernel pair $u,v:S\to M$ of $s$ in $\L$ and a weak reflection $t:M\rightarrowtail T$ associated to $M$ (which again we suppose to be a regular monomorphism):
	\begin{center}
		
		\begin{tikzpicture}[baseline=(current  bounding  box.south), scale=2]
		
		\node (a) at (-0.3,0) {$L$};
		\node (b) at (0.6,0) {$S$};
		\node (c) at (1.6,0) {$M$};
		\node (d) at (2.5,0) {$T$};

		\path[font=\scriptsize]
		
		(a) edge [>->] node [above] {$s$} (b)
		([yshift=1.5pt]b.east) edge [->] node [above] {$u$} ([yshift=1.5pt]c.west)
		([yshift=-1.5pt]b.east) edge [->] node [below] {$v$} ([yshift=-1.5pt]c.west)
		(c) edge [>->] node [above] {$t$} (d);
		\end{tikzpicture}
	\end{center}
	Then $t\circ u$ and $t\circ v$ define $L$ as an equalizer of elements from $\D$; call this a {\em presentation} for $L$. We are going to prove that these presentations are $J$-absolute, in the sense that they are sent to coequalizers by $\L(-,R)$ for each $R\in\D$. For, given $R\in\D$, consider the induced diagram
	\begin{center}
		
		\begin{tikzpicture}[baseline=(current  bounding  box.south), scale=2]
		
		\node (a) at (-1,0) {$\L(T,R)$};
		\node (b) at (0.5,0) {$\L(M,R)$};
		\node (c) at (2,0) {$\L(S,R)$};
		\node (d) at (3.5,0) {$\L(L,R)$.};

		\path[font=\scriptsize]
		
		(a) edge [->>] node [above] {$\L(t,R)$} (b)
		([yshift=1.5pt]b.east) edge [->] node [above] {$\L(u,R)$} ([yshift=1.5pt]c.west)
		([yshift=-1.5pt]b.east) edge [->] node [below] {$\L(v,R)$} ([yshift=-1.5pt]c.west)
		(c) edge [->>] node [above] {$\L(s,R)$} (d);
		\end{tikzpicture}
	\end{center}
	Then $\L(t,R)$ and $\L(s,R)$ are regular epimorphisms since $t$ and $s$ are weak reflections; while $\L(u,R)$ and $\L(v,R)$ form the kernel pair of $\L(s,R)$ since $\L(-,R)$ transforms colimits into limits. As a consequence $\L(s,R)$ is the coequalizer of $\L(t\circ u,R)$ and $\L(t\circ v,R)$ as desired. It follows then that, for each object $L$ of $\L$, the presentations we are considering are actually codensity presentations; hence $\D$ is codense in $\L$ by \cite[Theorem~5.19(v)]{Kel82:libro}. 
\end{proof}

\begin{obs}
	The Proposition still holds even with a weaker notion of weak reflection obtained replacing regular epimorphisms with just epimorphisms, the hypothesis on $\L$ could also be dropped: see \cite[Proposition~1.3.5]{Ten19:articolo}.
\end{obs}

\section{Finitary Varieties and Quasivarieties}\label{fqv}

Let us start this section by studying the main properties of categories with a strong generator consisting of regular projective objects. In this paper, we shall abbreviate regular projective to projective since no other notion of projectivity is considered. 

\begin{Def}
	Let $\K$ be a category; an object $P$ of $\K$ is called {\em projective} if the hom-functor $\K(P,-):\K\to \bo{Set}$ preserves all regular epimorphisms existing in $\K$; in other words, if $\K(P,-)$ sends regular epimorphisms to surjections. Denote by $\K_p$ the full subcategory of $\K$ given by the projective objects.
\end{Def}

\begin{lema}\cite[Lemma~2.1.4]{Ben89:articolo}\label{refregep}
	Let $\B$ be a regular category, $\A$ have finite limits and coequalizers of kernel pairs, and $F:\A\to\B$ be a functor which preserves finite limits and regular epimorphisms. Then the following are equivalent:\begin{enumerate}
		 
		\item $F$ is conservative;
		\item $F$ reflects regular epimorphisms.
	\end{enumerate}
	Furthermore $\A$ is then a regular category.
\end{lema}

\begin{prop}\label{locallyproj}
	Let $\K$ have finite limits and coequalizers of kernel pairs. The following are equivalent:\begin{enumerate}
		 
		\item $\K$ has a strong generator made of projective objects;
		\item there exists a small $\P\subseteq\K_p$ such that a morphism $f$ in $\K$ is a regular epimorphism if and only if $\K(P,f)$ is surjective for each $P\in\P$.
	\end{enumerate} 
	Furthermore, if they hold, $\K$ is a regular category.
\end{prop}
\begin{proof}
	Let $i:\P\to\K$ be any small full subcategory consisting of projective objects, and let $N:\K\longrightarrow [\P^{op},\bo{Set}]$ be the functor sending $A\in\K$ to $\K(i-,A)$. This preserves finite limits and regular epimorphisms. Now apply Lemma \ref{refregep}, recalling that $\P$ is a strong generator if and only if $N$ is conservative.
\end{proof}

Recall that an exact category is a regular one with effective equivalence relations (see \cite{Bar86:articolo}), these are sometimes called Barr-exact or effective regular (as in \cite{Joh02:libro}). We now recall the notions of finitary variety and quasivariety:

\begin{Def}
	A category $\K$ is called a {\em finitary quasivariety} if it is cocomplete and has a strong generator formed by finitely presentable projective objects (and so is regular). If moreover $\K$ is an exact category, it is called a {\em finitary variety}. Denote by $\K_{pf}$ the full subcategory of finitely presentable projective objects of $\K$, we are going to refer to them simply as {\em finite projective} objects.
\end{Def}

By \cite[Theorem~3.24]{AR94:libro} (or actually, the correction appearing in \cite{AR94corrections}), this corresponds to the usual definitions of multisorted finitary quasivariety and variety. In fact, finitary varieties can be described as the categories of models of multi-sorted algebraic theories (whose axioms are systems of linear equations); while finitary quasivarieties are the categories of models of theories whose axioms are implications of equations. 

Finitary varieties can also be described as follows:

\begin{teo}\cite[Theorem~3.16]{AR94:libro}
	A category $\K$ is a finitary variety if and only if it is equivalent to $\textnormal{FP}(\C,\bo{Set})$ for some small category $\C$ with finite products. In particular we could take $\C$ to be $(\K_{pf})^{op}$, which is in addition a Cauchy complete category.
\end{teo}

\begin{ese}$ $\begin{itemize}
		 
		\item $\bo{Set}$ and $\bo{Ab}$ are finitary varieties; we may take $\P$ to be $\{1\}$ and $\{\bo{Z}\}$ respectively.
		\item The category $\bo{BRel}$ of sets with a binary relation, is a \fqv\ (but not a finitary variety), with strong generator given by the singleton (with empty relation) and the doubleton (with nonempty irreflexive, antisymmetric relation).
		\item For any small $\A$, the functor category $[\A,\bo{Set}]$ is a finitary variety with strong generator $\P$ given by the set of all representable objects.
		\item If $(T,\mu,\eta)$ is a monad on a finitary variety (resp. quasivariety) $\K$ and $T$ preserves filtered colimits and regular epimorphisms, then the Eilenberg-Moore category $\K^T$ is a finitary variety (resp. quasivariety). A strong generator of $\K^T$ is given by the set of free algebras over finite projective objects of $\K$.
		\item All the examples from \cite[Example~3.20]{AR94:libro}.
	\end{itemize}
\end{ese}

Note that if $\K$ is a finitary quasivariety then it is locally finitely presentable (and so complete) as well as a regular category (by Proposition \ref{locallyproj}). As usual we write $\K_f$ for the full subcategory of $\K$ consisting of the finitely presentable objects. Recall also:

\begin{prop}\cite[Theorem~3]{AP98:articolo}
	A \lfp\ category $\K$ is a \fqv\ if and only if $(\K_f)^{op}$ has enough injectives.
\end{prop}

\begin{prop}\label{enoughproj}
	Let $\K$ be a \fqv\ with strong generator $\P\subseteq\K_{pf}$; then:
	\begin{enumerate}
		\item regular epimorphisms are closed under small products in $\K$;
		\item $\K_p$ is closed in $\K$ under small coproducts and retracts;
		\item $\K$ has enough projectives;
		\item $Q$ is in $\K_p$ if and only if it is a retract of a coproduct of objects of $\P$.
	\end{enumerate}
\end{prop}
\begin{proof}
	$(1)$. Let $(e_i)_{i\in I}$ be a set of regular epimorphisms in $\K$; then $\K(P,\prod_i e_i)\cong \prod_i\K(P,e_i)$ is a surjection for each $P\in\P$ (since surjections are product stable in $\bo{Set}$). As a consequence $\prod_i e_i$ is a regular epimorphism in $\K$ by Proposition \ref{locallyproj}.\\
	$(2)$. Consider a coproduct $\coprod_i P_i$ of projective objects, then $\K(\coprod_i P_i,-)\cong \prod_i\K(P_i,-)$ preserves surjections because these are product stable in $\bo{Set}$. It follows that $\coprod_i P_i$ is projective. For retracts, let $i:Q\rightarrowtail P$ be a retract of some $P\in\K_p$. Consider $p:P\to Q$ such that $p\circ i=id_Q$; then given a regular epimorphism $e:A\to B$ in $\K$ and $f:Q\to B$, since $P$ is projective there is $g':P\to A$ such that $e\circ g'=f\circ p$. Define then $g:= g'\circ i$; it is easy to see that $e\circ g=f$ and hence that $Q\in\K_p$.\\
	$(3)$. Let $K$ be an object of $\K$; since $\P$ is strongly generating and $\K$ is regular, there exists a regular epimorphism $P:=\coprod_i P_i\twoheadrightarrow K$, with $P_i\in\P$. But $\K_p$ is closed under coproducts, so $P\in\K_p$.\\
	$(4)$. Let $Q\in\K_p$, then as before there is a regular epimorphism $\coprod_i P_i\twoheadrightarrow Q$ with $P_i\in\P$ for each $i$. Since $Q$ is projective this regular epimorphism splits as desired. The converse follows by $(2)$.
\end{proof}

\begin{prop}
	Let $\V$ be a symmetric monoidal closed complete and cocomplete category, and let $\K$ be any cocomplete $\V$-category for which $\K_0$ is a finitary variety (resp. quasivariety). Then for any small $\V$-category $\A$, the category $[\A,\K]_0$ of $\V$-functors from $\A$ to $\K$ is a finitary variety (resp. quasivariety).
\end{prop}
\begin{proof}
	Let $\P\subseteq(\K_0)_p$ be a strong generator for $\K_0$ made of finite projective objects. Define $\P'$ in $[\A,\K]$ as the collection of those functors of the form $\A(a,-)\cdot P$ for each $a\in\A$ and $P\in\P$. These are projective since for any regular epimorphism $e$ in $[\A,\K]$ the following isomorphisms hold
	$$ [\A,\K]_0(\A(a,-)\cdot P,e)\cong[\A,\V]_0(\A(a,-),\K(P,e-))\cong \K_0(P,e_a) $$
	and the last is a regular epimorphism since $e_a$ is one and $P$ is projective. An analogous chain of isomorphisms shows that the elements of $\P'$ are finitely presentable (each $p\in\P$ is finitely presentable and evaluation at $d$ preserves all limits and colimits).\\
	It remains to prove that $\P'$ is a strong generator.
	Given $F$ in $[\A,\K]$, it's enough to prove that for any $d\in\A$ there are $P\in\P$ and $\eta:\A(d,-)\cdot P\to F$ such that $\eta_d$ is a regular epimorphism; because then we can just take the coproduct of those maps over $d\in\A$. Since $\K_0$ is locally projective, given $d$ there are $P\in\P$ and a regular epimorphism $f:P\twoheadrightarrow Fd$, define then $\eta$ as the natural transformation whose transpose $\bar{\eta}:\A(d,-)\to\K(P,F-)$ corresponds, through Yoneda, to $f$. Consider then the following diagram
	\begin{center}
		\begin{tikzpicture}[baseline=(current  bounding  box.south), scale=2]
		
		\node (a) at (0.0,0.8) {$I\cdot P$};
		\node (b) at (1.5,0.1) {$Fd$};
		\node (c) at (0.0,0.1) {$\A(d,d)\cdot P$};
		
		\path[font=\scriptsize]

		(a) edge[->>] node [above] {$f$} (b)
		(a) edge[->] node [left] {$id\cdot P$} (c)
		(c) edge[->] node [below] {$\eta_d$} (b);
		
		\end{tikzpicture}
	\end{center}
	since $\K_0$ is regular (by Proposition \ref{locallyproj}) and $f$ a regular epimorphism, $\eta_d$ is one too.
\end{proof}

\begin{es}
	It follows that, for each commutative ring $R$, the categories $R\text{-Mod}$ and GR-$R\text{-Mod}$ of $R$-modules and of $\mathbb{Z}$-graded $R$-modules, are finitary varieties. Moreover if $\A$ is abelian and a finitary variety (resp. quasivariety) then so is the category $\text{Ch}(\A)$ of chain complexes on $\A$.
\end{es}

With a similar approach to that of \cite{Kel82:articolo} for \lfp\ $\V$-categories, we define:

\begin{Def}\label{smfv}
	Let $\V=(\V_0,\otimes,I)$ be a symmetric monoidal closed category. We say that $\V$ is a {\em symmetric monoidal \fqv\ }if:\begin{enumerate}
		 
		\item $\V_0$ is a \fqv\ with strong generator $\P\subseteq(\V_{0})_{pf}$;
		\item $I\in(\V_{0})_f$;
		\item if $P,Q\in\P$ then $P\otimes Q\in(\V_{0})_{pf}$.
	\end{enumerate}
	We call it a {\em symmetric monoidal finitary variety} if $\V_0$ is also a finitary variety.
\end{Def}

In view of Remark \ref{0p=p0} and Propositions \ref{classsymmonfv} and \ref{I+closeddef}, one might think it would be reasonable to assume the unit $I$ to be projective in $\V_0$. However this is not needed to prove the main theorems of the paper, and there are significant examples of symmetric monoidal finitary varieties not satisfying the property (like chain complexes).

\begin{ese} The following are examples of symmetric monoidal finitary quasivarieties:
	\begin{enumerate}[(i)]
		
		\item $\bo{Set}$ and $\bo{Ab}$ with the cartesian and group tensor product respectively;
		\item $R$-Mod and GR-$R$-Mod, for each commutative ring $R$, with the usual algebraic tensor product;
		\item $[\C^{op},\bo{Set}]$, for any category $\C$ with finite products, equipped with the cartesian product; 
		\item the category $\bo{Set}_*$ of pointed sets with the smash product;
		\item the category $\bo{Set}^G$ of $G$-sets for a finite group $G$ with the cartesian product;
		\item the category $\bo{Gra}$ of directed graphs with the cartesian product;
		\item $\text{Ch}(\A)$ for each abelian and symmetric monoidal finitary quasivariety $\A$, with the tensor product inherited from $\A$;
		\item the category $\bo{Ab}_{tf}$ of torsion free abelian groups with the usual tensor product; 
		\item $\bo{BRel}$ with the cartesian product;
		\item the full subcategory $\bo{Mono}$ of all monomorphisms in $\bo{Set}^{\mathbbm{2}}$.
	\end{enumerate}
	The first four are always symmetric monoidal finitary varieties with projective units. Examples (v) and (vi) are also symmetric monoidal finitary varieties but the unit is not projective (except in (v) if $G$ is trivial). Example (vii) is a symmetric monoidal finitary variety if $\A$ is one, but once again the unit is generally not projective. The remaining examples are not symmetric monoidal finitary varieties; the unit is projective in (x). Non-examples are: $\bo{Cat}$ with any tensor product (since it is not a quasivariety); the categories $\bo{RGra}$ of reflexive graphs, and $\bo{sSet}$ of simplicial sets with the cartesian product (since the product of two projective objects may not be projective).  
\end{ese}

\begin{obs}
	Let $\V$ be a symmetric monoidal \fqv; then point $(3)$ of the previous definition implies, by \cite[Proposition~5.2]{Kel82:articolo}, that $\V_{0f}$ is closed under tensor product. The same holds for $\V_{0p}$: given two projective objects $P,Q\in\V_{0p}$, there are split monomorphisms $P\rightarrowtail \coprod_i P_i$ and $Q\rightarrowtail\coprod_j Q_j$, with $P_i,Q_j\in\P$. Then $P\otimes Q$ is a split subobject of $$\coprod_i P_i\otimes\coprod_i Q_i\cong\coprod_{i,j} (P_i\otimes Q_j),$$ which is projective; hence $P\otimes Q$ is projective. It follows then that $(\V_{0})_{pf}$ is also closed under tensor product (but may not contain the unit).
\end{obs}

The following proposition gives a characterization of the monoidal structures on a finitary variety that make it a symmetric monoidal finitary variety, assuming some additional conditions.

\begin{prop}\label{classsymmonfv}
	Let $\C$ be a Cauchy complete category with finite products; there is an equivalence between \begin{itemize}
		 
		\item symmetric monoidal structures on $\C$ for which $-\otimes -:\C\times\C\to\C$ preserves finite products in each variable;
		\item symmetric monoidal structures on $\textnormal{FP}(\C,\bo{Set})$ which make it a symmetric monoidal finitary variety with projective unit.  	
	\end{itemize}
	Moreover, the induced structures make the Yoneda embedding $Y:\C^{op}\to\textnormal{FP}(\C,\bo{Set})$ a strong monoidal functor.
\end{prop}
\begin{proof}
	On one side, since $\C$ is Cauchy complete, it follows that $\textnormal{FP}(\C,\bo{Set})_{pf}\simeq\C^{op}$ (finite projectives are split subobjects of representables); then the remark above implies that every symmetric monoidal structure on $\textnormal{FP}(\C,\bo{Set})$ which makes it a symmetric monoidal finitary variety with projective unit, restricts to a symmetric monoidal structure on $\C$. The functor $-\otimes -:\C\times\C\to\C$ preserves finite products in each variable since this is true in $\textnormal{FP}(\C,\bo{Set})$.
	
	On the other side, let $(\C,\otimes,I)$ be a symmetric monoidal structure on $\C$ as in the first point. It is proven in \cite{Day70:articolo} that it induces a symmetric monoidal closed structure on $[\C,\bo{Set}]$ for which the Yoneda embedding $\ \C^{op}\to[\C,\bo{Set}]$ is strong monoidal and for every $F,G:\C\to\bo{Set}$ and $c\in\C$
	$$ (F\otimes G) (c)\cong\int^{c_1,c_2\in\C}\C(c_1\otimes c_2,c)\times F(c_1)\times F(c_2)$$ 
	can be expressed as a coend.
	Now, if $F$ and $G$ preserve finite products, by \cite[Corollary~2.8]{AR11:articolo}, we can write them as sifted colimits of representables: $F\cong\text{colim}_iY(c_i)$ and $G\cong\text{colim}_jY(d_j)$. Since sifted colimits commute in $\bo{Set}$ with finite products and coends, it follows that $$F\otimes G\cong\text{colim}_{i}Y(c_i)\otimes \text{colim}_{j}Y(d_j)\cong\text{colim}_{i,j}Y(c_i\otimes d_j),$$ making $F\otimes G$ a sifted colimit of representables and hence a finite product preserving functor. As a consequence the tensor product on $[\C,\bo{Set}]$ restricts to $\textnormal{FP}(\C,\bo{Set})$, and satisfies conditions $(2)$ and $(3)$ of Definition \ref{smfv} (with $\P=Y\C^{op}$); we are only left to prove that the symmetric monoidal structure induced on $\textnormal{FP}(\C,\bo{Set})$ is closed. For this it's enough to show that if $F$ preserves finite products and $G$ is any functor, then the internal hom $[G,F]$ (seen in $[\C,\bo{Set}]$) preserves finite products too. Write $G\cong\text{colim}_jY(d_j)$ as a colimit of representables; then $[G,F]\cong[\text{colim}_jY(d_j),F]\cong\text{lim}_j[Y(d_j),F]$ and it suffices to show that $[Y(c),F]$ preserves finite products for every $c\in\C$. Fix $d\in\C$, then
	\begin{equation*}
	\begin{split}
	[Y(c),F](d)&\cong [\C,\bo{Set}](Y(d),[Y(c),F])\\
	&\cong [\C,\bo{Set}](Y(d)\otimes Y(c),F)\\
	&\cong [\C,\bo{Set}](Y(d\otimes c),F)\\
	&\cong F(d\otimes c);
	\end{split}
	\end{equation*}
	in other words $[Y(c),F]\cong F(-\otimes c)$, and this preserves finite products since $F$ does and $-\otimes -:\C\times\C\to\C$ preserves finite products in each variable by assumption (Observe that the tensor product on $\textnormal{FP}(\C,\bo{Set})$ coincides with that induced by Day's reflection \cite{Day72:articolo}).
\end{proof}

\begin{obs}\label{0p=p0}
	Note that if $\V$ is a symmetric monoidal \fqv\ it is in particular \lfp\ as a closed category (in the sense of \cite{Kel82:articolo}), and hence $\V_{0f}=\V_{f0}$. We can show similar results for the full subcategory of projectives. Denote with $\V_p$ the full subcategory of $\V$ given by the $\V$-projective objects: those $P\in\V$ such that $[P,-]:\V\to\V$ preserves regular epimorphisms. Then the inclusion $\V_{0p}\subseteq\V_{p0}$ holds; indeed given $P\in\V_{0p}$ and a regular epimorphism $e$, the function of sets $\V_0(Q,[P,e])\cong\V_0(Q\otimes P,e)$ is a surjection for each $Q\in\P$ (since $Q\otimes P$ is projective); hence $[P,e]$ is a regular epimorphism (by Proposition \ref{locallyproj}) and $P\in\V_{p0}$ (this means exactly that regular epimorphisms are stable in $\V$ under projective powers). The inclusion $\V_{p0}\subseteq\V_{0p}$ holds if and only if $I\in\V_{0p}$; indeed if  $I\in\V_{0p}$, given any $P\in\V_{p0}$, the functor $\V_0(P,-)\cong\V_0(I,[P,-])$ preserves regular epimorphisms; thence $P\in\V_{0p}$. Vice versa, if $\V_{0p}=\V_{p0}$, since $I\in\V_{p0}$ always ($[I,-]\cong id$), it follows that $I$ is projective in the ordinary sense.
\end{obs}

\section{Regular $\V$-Categories}\label{regcat}

From now on we assume that our base category $\V$ is a symmetric monoidal \fqv\ with strong generator $\P\subseteq (\V_0)_{pf}$ made of finite projective objects.

The following is the notion of regular $\V$-category we are going to consider in this context:

\begin{Def}
	A $\V$-category $\C$ is said to be {\em regular} if it has all finite weighted limits (equivalently finite conical limits and finite powers), coequalizers of kernel pairs, and is such that regular epimorphisms are stable under pullback and closed under powers by elements of $\P$.\\ 
	A $\V$-functor $F:\C\to\D$ between regular $\V$-categories is called {\em regular} if it preserves finite weighted limits and regular epimorphisms; we denote by $\textnormal{Reg}(\C,\D)$ the $\V$-category of regular functors from $\C$ to $\D$. 
\end{Def}

We define $\V$-$\bo{Reg}$ to be the 2-category of small regular $\V$-categories, regular $\V$-functors, and $\V$-natural transformations.

\begin{obs}\label{Chi}
	A different notion of regularity appeared before in \cite{Chi11:articolo}; there, in a regular $\V$-category, regular epimorphisms need to be stable under all finite powers, instead of just finite projective ones like in our case. At the same time the base for enrichment can be assumed to be only \lfp\ as a closed category, and one can still prove the analogue of \ref{Barr}. We chose to consider a different approach to recover the usual notions of regularity and exactness for $\V=\bo{Ab}$; in fact $\bo{Ab}$ itself is not regular as an additive category in the sense of \cite{Chi11:articolo}, but it is regular in our sense.
\end{obs}

It follows from the definition that a $\V$-category $\C$ is regular if and only if it has all finite weighted limits, $\C_0$ is an ordinary regular category, and regular epimorphisms are stable under powers with elements of $\P$. Indeed, this is easily checked to be necessary; on the other hand it is sufficient because, since $\C$ has $\P$-powers and $\P$ is a strong generator, coequalizers of kernel pairs in $\C$ and $\C_0$ coincide.

\begin{obs}
	$\V$ itself is regular as a $\V$-category since it is both complete and cocomplete, $\V_0$ is regular in the ordinary sense by Proposition \ref{locallyproj}, and regular epimorphisms are stable under all projective powers (by Remark \ref{0p=p0}).
\end{obs}

The condition on the stability of regular epimorphisms under powers from $\P$ is additional to the usual ordinary notion of regularity, but it must be included if we want to regularly embed each regular $\V$-category in a $\V$-category of presheaves. In fact, regular epimorphisms are stable under powers from $\P$ in any $\V$-category of the form $[\C,\V]$, since this is true in $\V$, and hence the same holds for any full subcategory of $[\C,\V]$ closed under finite weighted limits and coequalizers of kernel pairs.

\begin{obs}\label{reg/lex}
	Our notion of regular $\V$-category is a particular case of what are called $\Phi$-exact $\V$-categories in \cite{GL12:articolo}, $\Phi$ being a class of weights. More precisely, there is a suitable choice of $\Phi$ for which a $\V$-category is regular if and only if it is $\Phi$-exact; this follows by our embedding Theorem \ref{Barr} and \cite[Theorem~4.1]{GL12:articolo}. As a consequence, by Corollary 3.7 of the same paper, it follows that every finitely complete $\V$-category $\C$ has a free regular completion $\C_{\text{reg/lex}}$; meaning that there is a lex functor $F:\C\to\C_{\text{reg/lex}}$ which induces an equivalence $\text{Reg}(\C_{\text{reg/lex}},\B)\simeq\text{Lex}(\C,\B)$ for each regular $\V$-category $\B$. See also Remark \ref{reexreg} and Section \ref{free}.
\end{obs}

The following result follows from the $\bo{Set}$-case since each regular $\V$-category has an underlying ordinary regular category.

\begin{prop}[\cite{BGO71:libro}]
	Let $\C$ be a regular $\V$-category; then:\begin{enumerate}
		 
		\item each morphism $f$ in $\C$ can be factored as a regular epimorphism followed by a monomorphism; the factorization is unique up to unique isomorphism;
		\item regular and strong epimorphisms coincide in $\C$;
		\item if $f$ and $g$ are regular epimorphisms then $f\circ g$ is too;
		\item if $f=g\circ h$ is a regular epimorphism, then $g$ is too;
		\item regular epimorphisms are stable under finite products.
	\end{enumerate}
\end{prop}

With the next proposition we show that the definition of regularity does not depend on the chosen strong generator $\P$ in $(\V_0)_{pf}$:

\begin{prop}
	Let $\C$ be a regular $\V$-category; then regular epimorphisms are stable in $\C$ under powers with each element of $(\V_0)_{pf}$.
\end{prop}
\begin{proof}
	Let $h:A\to B$ be a regular epimorphism in $\C$ and $P\in(\V_0)_{pf}$. By Proposition \ref{enoughproj}, $P$ is a split subobject of a coproduct $Q:=\coprod_iP_i$ with $P_i\in\P$; write $m:P\to Q$ for the split monomorphism. Since $P$ is also finitely presentable, we can assume the coproduct to be finite; as a consequence $Q$ is finitely presentable and $h^Q$ exists in $\C$. Moreover $h^Q\cong \prod_i (h^{P_i})$ is a regular epimorphism since the $h^{P_i}$ are, and regular epimorphisms are stable under finite products in each ordinary regular category.
	Consider then the square
	\begin{center}
		\begin{tikzpicture}[baseline=(current  bounding  box.south), scale=2]
		
		\node (a) at (0,0.8) {$A^Q$};
		\node (b) at (1,0.8) {$B^Q$};
		\node (c) at (0, 0) {$A^P$};
		\node (d) at (1, 0) {$B^P$};
		
		\path[font=\scriptsize]
		
		(a) edge [->>] node [above] {$h^Q$} (b)
		(a) edge [->>] node [left] {$A^m$} (c)
		(b) edge [->>] node [right] {$B^m$} (d)
		(c) edge [->] node [below] {$h^P$} (d);
		
		\end{tikzpicture}	
	\end{center}
	where $A^m$ and $B^m$ are split epimorphisms, and hence regular. As a consequence, since $\C_0$ is regular, it follows that $h^P$ is a regular epimorphism as desired.
\end{proof}

\begin{Def}
	A $\V$-category $\B$ is called exact if it is regular and in addition the ordinary category $\B_0$ is exact in the usual sense.
\end{Def}

Taking $\V=\bo{Set}$ or $\V=\bo{Ab}$ this notion coincides with the ordinary one of exact or abelian category. Note moreover that our base $\V$ may not be exact (but only regular).

\begin{obs}\label{reexreg}
	If $\V$ is a symmetric monoidal finitary variety, then $\V_0$ is an ordinary exact category and $\V$ is exact as a $\V$-category. Arguing as in Remark \ref{reg/lex}, it's easy to see that our notion of exactness then coincides with that of $\Phi'$-exactness for a suitable $\Phi'$ (different from that defining regularity). It follows then \cite[Theorem~7.7]{GL12:articolo} that each regular $\V$-category has an exact completion $\C_{\text{ex/reg}}$. Similarly each finitely complete $\V$-category $\C$ has an exact completion $\C_{\text{ex/lex}}$. These will be described explicitly in Section \ref{free}.
\end{obs}

\section{Definable $\V$-Categories}\label{def}

We consider again categories enriched over a base $\V$ which is a symmetric monoidal \fqv\ with strong generator $\P\subseteq (\V_0)_{pf}$.

The following is the corresponding enriched version of Definition \ref{ofic}:

\begin{Def}
	Given an arrow $h:A\to B$ in a $\V$-category $\L$, an object $L\in\L$ is said to be $h${\em -injective} if $\L(h,L):\L(B,L)\to\L(A,L)$ is a regular epimorphism in $\V$. Given a small set $\M$ of arrows from $\L$ write $\M$-inj for the full subcategory of $\L$ consisting of those objects which are $h$-injective for each $h\in\M$. $\V$-categories arising in this way are called {\em enriched injectivity classes}, or just {\em injectivity classes} if no confusion will arise. If $\L$ is \lfp\ and the arrows in $\M$ have finitely presentable domain and codomain, we call $\M$-inj an {\em enriched finite injectivity class}. 
\end{Def}

Injectivity classes in the enriched context were first considered in \cite{LR12:articolo}. In that setting a more general notion is introduced: regular epimorphisms are replaced by a suitable class $\E$ of morphisms from $\V$. An object $L$ was called $\E$-injective if $\L(h,L)\in\E$, and an $\E$-injectivity class was the full subcategory of $\E$-injective objects with respect to a small set of morphisms. Among others was studied the case in which $\V$ was \lfp\ as a closed category and $\E$ the class of pure epimorphisms of $\V$, obtaining a characterization of enriched $\E$-injectivity classes \cite[Theorem~8.8 and 8.9]{LR12:articolo}. \\
If we assume the unit of $\V_0$ to be projective, exactly the same arguments apply to our case (replacing finite powers by projective ones, and using Proposition \ref{locallyproj}); hence we obtain two corresponding characterization theorems.

Before stating them, recall that a $\V$-category $\A$ is called {\em accessible} if there exists a regular cardinal $\lambda$ such that $\A$ has (conical) $\lambda$-filtered colimits and a small set of $\lambda$-presentable objects which generates $\A$ under $\lambda$-filtered colimits.

\begin{teo}\label{inj-char} Let $\L$ be a locally presentable $\V$-category, and $\A$ a full subcategory of $\L$. Consider the following conditions:\begin{enumerate}
		\item $\A$ is equivalent to a small injectivity class in $\L$;
		\item $\A$ is accessible, accessibly embedded, and closed under products and projective powers in $\L_0$;
		\item $\A$ is accessibly embedded and weakly reflective (in the sense of Definition \ref{weakref}).
	\end{enumerate}
Then $(1)\Rightarrow (2)\Rightarrow (3)$, while $(3)\Rightarrow (1)$ if the unit $I$ of $\V_0$ is projective.
\end{teo}

\begin{teo} Assume the unit $I$ of $\V_0$ to be projective. For a $\V$-category $\A$, the following are equivalent:
	\begin{enumerate}
		\item $\A$ is a weakly reflective, accessibly embedded subcategory of $[\C,\V]$ for some small $\V$-category $\C$;
		\item $\A$ is equivalent to a small injectivity class in some locally presentable $\V$-category;
		\item $\A$ is accessible and has products and projective powers;
	\end{enumerate}
	
\end{teo}

In this paper we focus in particular on {\em finite} injectivity classes, for which such characterization theorems are more difficult to obtain. It's easy to see that for each enriched (finite) injectivity class $\D$ the underlying category $\D_0$ is an ordinary (finite) injectivity class: indeed, if $\D=\M$-inj in $\L$, then $\D_0=\M_0$-inj in $\L_0$ where 
$$\M_0=\{ P\cdot h\ |\ P\in\P,\ h\in\M \},$$
this because $\L(h,S)$ is a regular epimorphism in $\V$ if and only if $\L_0(P\cdot h,S)\cong\V_0(P,\L(h,S))$ is surjective for each $P\in\P$ (by Proposition \ref{locallyproj}).

In the ordinary case each finite injectivity class is known to be closed under pure subobjects inside its \lfp\ category \cite[Theorem~2.2]{RAB02:articolo}; let us recall the definition and introduce a new notion we use in the enriched context.

\begin{Def}
	Let $f:X\to Y$ be a morphism in an ordinary \lfp\ category $\L$. We say that $f$ is {\em pure} if for each commutative square
	\begin{center}
		\begin{tikzpicture}[baseline=(current  bounding  box.south), scale=2]
		
		\node (a) at (0,0.8) {$A$};
		\node (b) at (0.8,0.8) {$B$};
		\node (c) at (0, 0) {$X$};
		\node (d) at (0.8, 0) {$Y$};
		
		\path[font=\scriptsize]
		
		(a) edge [->] node [above] {$h$} (b)
		(a) edge [->] node [left] {$u$} (c)
		(b) edge [->] node [right] {$v$} (d)
		(c) edge [->] node [below] {$f$} (d);
		
		\end{tikzpicture}	
	\end{center}
	with $h:A\to B$ in $\L_f$, there exists $l:B\to X$ such that $l\circ h=u$.\\
	If $\L$ is a \lfp\ $\V$-category, we say that $f:X\to Y$ in $\L$ is {\em $\P$-pure} if $f^P$ is pure for each $P\in\P$.
\end{Def}

The notion of purity we are considering is the ordinary one: a morphism $f$ in the \lfp\ $\V$-category $\L$ is pure just when it is pure in the ordinary category $\L_0$. Since the copower of a finitely presentable object of $\L$ by an object in $\P$ is still finitely presentable, any pure morphism in $\L$ is $\P$-pure. If $I$ is projective then any $\P$-pure object is pure, but in general this need not be the case: see Example \ref{G-pure}.

It is shown in \cite[Proposition~2.29]{AR94:libro} that each pure morphism (in a locally finitely presentable category) is actually a monomorphism, so that we can talk about pure subobjects. The same result easily follows also for $\P$-pure morphisms: if $f$ is $\P$-pure then each $f^P$ is a monomorphism (being pure); so each $\L(X,f^P)$ is a monomorphism, and hence each $\V_0(P,\L(X,f))$ is one. But $\P$ is a strong generator, so each $\L(X,f)$ is a monomorphism, so finally $f$ is a monomorphism.

\begin{prop}\label{ficclose}
	Each finite injectivity class $\D=\M\textnormal{-inj}$ of a \lfp\ $\V$-category $\L$ is closed under (small) products, projective powers (meaning that if $S\in\D$ and $P\in\V_{0p}$, then $S^P\in\D$), filtered colimits, and $\P$-pure subobjects (if $f:X\to Y$ is $\P$-pure and $Y\in\D$ then $X\in\D$).
\end{prop}
\begin{proof}
	Given any arrow $h\in\M$ and any object $L\in\L_0$, we can see $\L(h,L)$ as an object of the category of arrows $\V_0^2$; since the domain and codomain of $h$ are finitely presentable, the hom-functor $\L(h,-)_0:\L_0\to\V_0^2$ preserves filtered colimits as well as products and projective powers (since it preserves all limits). Note moreover that regular epimorphisms are stable in $\V$ under filtered colimits, products, and projective powers (as we saw in Section \ref{fqv}). As a consequence, if $S=\text{colim}_iS_i$ is a filtered colimit of objects of $\D$, then $\L(h,S)\cong \text{colim}_i\L(h,S_i)$ is a regular epimorphism; hence $S\in\D$. The same applies if $S$ is a product or a projective power of elements from $\D$.\\
	Suppose now that $Y$ is $\M$-injective, $X\in\L$, and $f:X\to Y$ is $\P$-pure. We are to show that $X$ is also $\M$-injective; or in other words, that $\L(h,X)$ is a regular epimorphism for all $h\in\M$, or that $\V_0(P,L(h,X))$ is a regular epimorphism for all $h\in\M$ and $P\in\P$; or that $X^P$ is injective in the ordinary sense with respect to $h$. But $f^P:X^P\to Y^P$ is pure, and $Y^P$ is injective with respect to $h$, so this follows by the ordinary case.
\end{proof}

We are now ready to introduce the notion of definable $\V$-category, which generalizes that introduced in \cite{Pre11:libro} for $\V=\bo{Ab}$, and in \cite{KR18:articolo} for $\V=\bo{Set}$.

\begin{Def}\label{defdef}
	A full subcategory of a \lfp\ $\V$-category $\L$ is called a {\em definable subcategory} of $\L$ if it is an enriched finite injectivity class of $\L$. A $\V$-category $\D$ is called {\em definable} if it is equivalent to a definable subcategory of some \lfp\ $\V$-category.
	A {\em definable functor} between definable $\V$-categories is a $\V$-functor that preserves products, projective powers, and filtered colimits. Denote by $\V\text{-}\bo{DEF}$ the 2-category of definable $\V$-categories, definable functors, and $\V$-natural transformations; we may allow ourselves to write $\textnormal{DEF}(\D,\D')$ for $\V\textnormal{-DEF}(\D,\D')$ when it is clear that we are dealing with definable $\V$-categories.
\end{Def}

It follows that, if $\D$ is a definable subcategory of $\L$ as above, then the underlying category $\D_0$ of $\D$ is an ordinary definable subcategory of $\L_0$.

\begin{obs}\label{regdef}
	Each \lfp\ $\V$-category is definable; moreover for any small regular $\V$-category $\C$, the $\V$-category $\text{Reg}(\C,\V)$ is a definable subcategory of $\text{Lex}(\C,\V)$. Indeed, $\text{Lex}(\C,\V)$ is \lfp\ and $\text{Reg}(\C,\V)=\M$-inj in $\text{Lex}(\C,\V)$ where $$\M:=\{\C(h,-)\ |\ h \text{ regular epimorphism in } \C\}.$$ Indeed a lex functor $F\in\text{Lex}(\C,\V)$ is in $\text{Reg}(\C,\V)$ if and only if $Fh$ is a regular epimorphism for each regular epimorphism $h$ in $\C$. But $Fh\cong\text{Lex}(\C,\V)(\C(h,-),F)$; hence $F$ is regular if and only if it is injective with respect to $\C(h,-)$ for any regular epimorphism $h$ in $\C$.
\end{obs}

In the ordinary case, closure under the three constructions in Proposition \ref{ficclose} is enough to characterize definable subcategories; indeed it is proven in \cite[Theorem~2.2]{RAB02:articolo} that a full subcategory $\D$ of a \lfp\ category $\L$ is a finite injectivity class if and only if it is closed in $\L$ under products, filtered colimits and pure subobjects (powers are not necessary since they are a special kind of products). We can obtain a similar result in this context with a further assumption.

\begin{prop}\label{I+closeddef}
	Assume the unit $I$ of $\V_0$ to be projective. Let $\D$ be a full subcategory of a \lfp\ $\V$-category $\L$; then $\D$ is a definable subcategory of $\L$ if and only if it is closed in $\L$ under products, projective powers, filtered colimits, and pure subobjects.
\end{prop}
\begin{proof}
	One direction is given by Proposition \ref{ficclose}. For the other, assume that $\D$ is closed in $\L$ under products, projective powers, filtered colimits, and pure subobjects. By \cite[Theorem~2.2]{RAB02:articolo}, $\D_0$ is an ordinary finite injectivity class in $\L_0$. Let $\M$ be the set of arrows defining $\D$ as such; we prove that it also defines $\D$ as an enriched finite injectivity class, and hence a definable subcategory of $\L$. Given $S\in\D$, $\L_0(h,S)$ is surjective for each $h\in\M$; but $\D$ is closed under projective powers, hence $\V_0(P,\L(h,S))\cong \L_0(h,S^P)$ is surjective for each $P\in\P$; thus $\L(h,S)$ is a regular epimorphism and $S\in\M$-inj. Conversely, given $S\in\M$-inj, $\L_0(h,S)=\V_0(I,\L(h,S))$ is surjective since $I\in\V_{0p}$, and as a consequence, $S\in\D$.
\end{proof}

This doesn't hold if $I$ is not projective, as explained in the following example:

\begin{es}\label{G-pure}
	Let $G$ be a non-trivial finite group and let $\V$ be the cartesian closed category $\bo{Set}^G$ of $G$-sets. A strong generator that makes $\bo{Set}^G$ a symmetric monoidal finitary variety is the unique representable object, which is just $G$ with the regular action. A morphism $f:X\to Y$ in $\bo{Set}^G$ is $G$-pure if and only if it is a monomorphism, and either both $X$ and $Y$ are empty or neither is. In this context $G$-purity is strictly weaker than usual purity; for example the inclusion $G\to G+1$ is $G$-pure but not pure.
	
	Consider now the full subcategory $\D$ of $\bo{Set}^G$ consisting of those $G$-sets that have a fixed point. This is closed under products, filtered colimits, projective powers and pure subobjects, but not under $G$-pure subobjects ($G+1\in\D$ and $G\to G+1$ is $G$-pure, but $G\notin\D$). It follows that $\D$ is not an enriched finite injectivity class in $\bo{Set}^G$.
\end{es}

By Proposition \ref{ficclose}, a natural question would then be: if $\L$ is a \lfp\ $\V$-category and $\D$ is a full subcategory closed under products, filtered colimits, projective powers, and $\P$-pure subobjects, is it an enriched finite injectivity class? 
We still don't have an answer for this in general.

\section{Barr's Embedding Theorem}\label{BarrEmb}

	Let us fix a small regular $\V$-category $\C$ and consider $\textnormal{Reg}(\C,\V)$ as a full subcategory of $\textnormal{Lex}(\C,\V)$.

\begin{lema}\label{coreg}
	$\textnormal{Lex}(\C,\V)$ is a coregular $\V$-category.
\end{lema}
\begin{proof}
	Let $\L:=\textnormal{Lex}(\C,\V)$; it is shown in \cite[Theorem~3]{Chi11:articolo} that regular monomorphisms are stable in $\L$ under pushouts (note that the notion of regular category appearing in the cited paper is different from ours, but the same proof applies to this setting); this is done by proving that each pushout diagram with one specified arrow a regular monomorphism can be written as a filtered colimit of representable diagrams of the same kind. As a consequence we only need to show that if $H$ is a regular monomorphism in $\L$ then so is $P\cdot H$ for each $P\in\P$. The same argument used for pushout diagrams shows (see again \cite[Theorem~3]{Chi11:articolo}) that each regular monomorphism of $\L$ is a filtered colimit of regular monomorphisms between representables; since $P\cdot -$ preserves colimits it is then enough to consider $H=\C(h,-)$ for a regular epimorphism $h$ in $\C$. Now, the restricted Yoneda embedding $\C\to\L^{op}$ preserves finite limits; hence $P\cdot \C(h,-)\cong\C(h^P,-)$ for each $P\in\P$. But $\C$ is regular, therefore $h^P$ is a regular epimorphism and as a consequence, $P\cdot \C(h,-)$ is a regular monomorphism as claimed.
\end{proof}

\begin{lema}\label{regfic}
	$\textnormal{Reg}(\C,\V)$ is a definable and a weakly reflective subcategory of $\textnormal{Lex}(\C,\V)$ and the weak reflections can be chosen to be regular monomorphisms.
\end{lema}
\begin{proof}
	$\D:=\text{Reg}(\C,\V)$ is definable and weakly reflective in $\L:=\text{Lex}(\C,\V)$ by Remark \ref{regdef} and Theorem \ref{inj-char}. The class defining $\D$ as a finite injectivity class in $\L$ is given by $\M:=\{\C(h,-)\ |\ h \text{ regular epimorphism in } \C\}$; hence $\D_0$ is defined by $$\M_0:=\{P\cdot \C(h,-)\ |\ h \text{ regular epimorphism in } \C,\ P\in\P \}.$$
	Moreover the weak reflections can be chosen to be in the closure of $\M_0$ under transfinite composition and pushouts, as explained in Theorem \ref{wr} (this gives ordinary weak reflections, but using projective powers it's easy to see that they are actually enriched).
	Now, since $\L$ is coregular by Lemma \ref{coreg}, the elements $P\cdot\C(h,-)$ in the class defining $\D_0$ are regular monomorphisms. Moreover, the fact that $\L$ is \lfp\ (since $\C$ is finitely complete) and coregular, implies that filtered colimits commute in $\L$ with finite limits, and regular monomorphisms are stable under pushouts. Thus, by Theorem \ref{wr}, our weak reflections can actually be chosen to be regular monomorphisms. 
\end{proof}

This allows us to prove an enriched version of Barr's Embedding Theorem. In the following result $[\textnormal{Reg}(\C,\V),\V]$ will not in general be a $\V$-category, since the limits defining the hom-objects are too large to exist in $\V$. It will nonetheless be a $\V'$-category for some enlargement $\V'$ of $\V$ --- see \cite[Section~2.6]{Kel82:libro}. When we assert that $\textnormal{ev}_{\C}:\C\to[\textnormal{Reg}(\C,\V),\V]$ is regular, this should be understood to mean that it sends finite limits to pointwise ones, and similarly for coequalizers of kernel pairs. The image of the fully faithful functor $\textnormal{ev}_{\C}:\C\to[\textnormal{Reg}(\C,\V),\V]$ will of course be a $\V$-category.

\begin{teo}[Barr's Embedding]\label{Barr}
	Let $\C$ be a small regular $\V$-category; then the evaluation functor $\textnormal{ev}_{\C}:\C\to[\textnormal{Reg}(\C,\V),\V]$ is fully faithful and regular.
\end{teo}
\begin{proof}
	Let $\D=\text{Reg}(\C,\V)$ and $\L=\text{Lex}(\C,\V)$ as before. It follows by the previous Lemma that $\D$ is a weakly reflective subcategory of $\L$ and the weak reflections can be chosen to be regular monomorphisms. Hence we can apply Proposition \ref{codense} to deduce that $\D$ is codense in $\L$. To conclude, note that the functor $\text{ev}_{\C}$ is given by the composite of the restricted Yoneda embedding $Y:\C\to\L^{op}$ and of $\widehat{J}:\L^{op}\to [\D,\V]=[\text{Reg}(\C,\V),\V]$ where $\widehat{J}=\L(1,J)$ ($J$ being the inclusion of $\D$ in $\L$); the first is always fully faithful and the second is so because $\D$ is codense in $\L$. Hence $\text{ev}_{\C}$ is fully faithful.
\end{proof}

Similarly, for each small regular category $\C$, we can find a regular embedding of $\C$ into a category of presheaves over a small base:

\begin{teo}\label{functembedd}
	Let $\C$ be a small regular $\V$-category. Then there exists a small $\V$-category $\A$ and a fully faithful and regular functor $F:\C\to[\A,\V]$.
\end{teo}
\begin{proof}
	Let $Y:\C\to\L^{op}=\text{Lex}(\C,\V)^{op}$ be the codomain restriction of the Yoneda embedding. For each representable functor $L=\C(C,-)$ in $\L$ we can consider an equalizer
	\begin{center}
		
		\begin{tikzpicture}[baseline=(current  bounding  box.south), scale=2]
		
		\node (a) at (-0.3,0) {$L$};
		\node (b) at (0.6,0) {$S_L$};
		\node (c) at (1.6,0) {$T_L$};

		\path[font=\scriptsize]
		
		(a) edge [>->] node [above] {$s$} (b)
		([yshift=1.5pt]b.east) edge [->] node [above] {$u$} ([yshift=1.5pt]c.west)
		([yshift=-1.5pt]b.east) edge [->] node [below] {$v$} ([yshift=-1.5pt]c.west);
		\end{tikzpicture}
	\end{center}
	where $S_L$ and $T_L$ are in $\text{Reg}(\C,\V)$, and $s$ is a weak reflection of $L$ into $\text{Reg}(\C,\V)$ (take $u',v':S_L\to M$ to be the cokernel pair of $s$, then define $u=s'\circ u'$ and $v=s'\circ v'$ where $s'$ is a weak reflection of $M$). Consider then the full subcategory $\B$ of $\L$ given by the representable functors and, for each of them, two regular functors $S_L$ and $T_L$ with the property just described. Let $\A\subset \B$ be the full subcategory consisting of all regular functors in $\B$. Then $\A$ is weakly reflective in $\B$ and the weak reflection can be chosen to be regular monomorphisms (if $B\in\B$ is representable, then this is true by construction; if $B$ is one of the new objects, then $B\in\A$ and the identity map is a weak reflection). By construction $\A$ and $\B$ are small categories, and, applying the same proof of Proposition \ref{codense},we obtain that $\A$ is codense in $\B$.
	Write $Y':\C\to\B^{op}$ for the codomain restriction of $Y$; then we can consider the functor $F:\C\to[\A,\V]$ defined as the composite
	\begin{center}
		
		\begin{tikzpicture}[baseline=(current  bounding  box.south), scale=2]
		
		\node (a) at (-0.3,0) {$\C$};
		\node (b) at (0.6,0) {$\B^{op}$};
		\node (c) at (1.8,0) {$[\A,\V]$};

		\path[font=\scriptsize]
		
		(a) edge [->] node [above] {$Y'$} (b)
		(b) edge [->] node [above] {$\B(1,J)$} (c);
		\end{tikzpicture}
	\end{center}
	where $J:\A\to\B$ is the inclusion. $F$ turns out to be just the evaluation functor restricted to $\A$; thence, since $\A\subseteq \text{Reg}(\C,\V)$, the functor $F$ is regular too. Finally $F$ is fully faithful because $Y'$ is, and $\A$ is codense in $\B$.
\end{proof}

\section{Makkai's Image Theorem}\label{MakIm}

Given any regular $\V$-category $\C$ we can consider the fully faithful functor
$$\textnormal{ev}_{\C}:\C\to[\textnormal{Reg}(\C,\V),\V]$$
given by Theorem \ref{Barr}. Moreover $\textnormal{Reg}(\C,\V)$ is closed in $[\C,\V]$ under products, projective powers and filtered colimits: this follows from Proposition \ref{ficclose} plus the fact that $\text{Lex}(\C,\V)$ is closed in $[\C,\V]$ under the same limits and colimits.  

It's then easy to see that the essential image of $\textnormal{ev}_{\C}$ is contained in $\textnormal{DEF}(\textnormal{Reg}(\B,\V),\V)$. We are going to see that, if $\C$ is moreover exact, this will actually be the essential image of $\textnormal{ev}_{\C}$.

The case $\V=\bo{Set}$ of the following Lemma appears as the first part of 3.2.2 in \cite{KR18:articolo}.

\begin{lema}\label{regepi}
	Let $\L$ be a \lfp\ $\V$-category and $J:\D\hookrightarrow\L$ an enriched finite injectivity class of $\L$. Given $F\in\textnormal{DEF}(\D,\V)$, suppose that for each $L\in\D$ and $x:I\to FL$ there exist $A\in\L_f$ and $\eta:\L(A,J-)\to F$ such that $x$ factors through $\eta_L$. Then there is a $B\in\L_f$ and a pointwise regular epimorphism $\L(B,J-)\twoheadrightarrow F$.
\end{lema}
\begin{proof}
	By Theorem \ref{inj-char} $\D$ is an accessible category. Consider then a regular cardinal $\lambda$ such that $\D$ is $\lambda$-accessible, and denote by $\D_{\lambda}$ the full subcategory of $\lambda$-presentable objects in $\D$.
	For each $S\in\D_{\lambda}$ take $P_S\in\V_{0p}$ and a regular epimorphism $\bar{x}_S:P_S\twoheadrightarrow FS$; this corresponds to an arrow $x_S:I\to (FS)^{P_S}\cong F(S^{P_S})$.\\
	Define $\widehat{S}:=\prod_{S\in\D_{\lambda}}S^{P_S}$ with projection maps $\pi_S:\widehat{S}\to S^{P_S}$; then $\widehat{S}\in\D$ and, since $F$ preserves products and projective powers, $F(\widehat{S})\cong\prod_{S\in\D_{\lambda}}F(S)^{P_S}$. Consider then $x:I\to F\widehat{S}$ with components $x_S:I\to F(S)^{P_S}$ for each $S\in\D_{\lambda}$. By our assumptions there exist $A\in\L_f$ and a natural transformation $\eta:\L(A,J-)\to F$ such that $x=\eta_{\widehat{S}}\circ y$ for some $y:I\to\L(A,\widehat{S})$.	For each $S\in\D_{\lambda}$ we can consider the following diagram:
	\begin{center}
		\begin{tikzpicture}[baseline=(current  bounding  box.south), scale=2]
		
		\node (0) at (0.7,1.5) {$I$};
		\node (a) at (0,0.8) {$\L(A,\widehat{S})$};
		\node (b) at (1.4,0.8) {$F\widehat{S}$};
		\node (c) at (0, 0) {$\L(A,S^{P_S})$};
		\node (d) at (1.4, 0) {$F(S^{P_S})$};
		\node (e) at (0, -0.8) {$\L(A,S)^{P_S}$};
		\node (f) at (1.4, -0.8) {$F(S)^{P_S}$};
		
		\path[font=\scriptsize]
		(0) edge [->] node [above] {$y$} (a)
		(0) edge [->] node [above] {$x$} (b)
		(a) edge [->] node [above] {$\eta_{\widehat{S}}$} (b)
		(b) edge [->] node [right] {$F\pi_S$} (d)
		(a) edge [->] node [left] {$\L(A,\pi_S)$} (c)
		(c) edge [->] node [below] {$\eta_{(S^{P_S})}$} (d)
		(c) edge [->] node [left] {$\cong$} (e)
		(d) edge [->] node [right] {$\cong$} (f)
		(e) edge [->] node [below] {$(\eta_S)^{P_S}$} (f);
		\end{tikzpicture}	
	\end{center}
	Transposing the vertical arrows then we obtain maps $\bar{y}_S:P_S\to\L(A,S)$ such that the diagram
	\begin{center}
		\begin{tikzpicture}[baseline=(current  bounding  box.south), scale=2]
		
		\node (0) at (0.6,1.5) {$P_S$};
		\node (a) at (0,0.8) {$\L(A,S)$};
		\node (b) at (1.2,0.8) {$FS$};

		\path[font=\scriptsize]
		(0) edge [->] node [left] {$\bar{y}_S$} (a)
		(0) edge [->>] node [right] {$\bar{x}_S$} (b)
		(a) edge [->] node [above] {$\eta_S$} (b);
		\end{tikzpicture}	
	\end{center}
	commutes.\\
	Since $\bar{x}_S$ is a regular epimorphism, $\eta_S$ is a regular epimorphism too for each $S\in\D_{\lambda}$ (remember that $\V_0$ is regular by Proposition \ref{locallyproj}), but $\D_{\lambda}$ generates $\D$ under $\lambda$-filtered colimits and $\L(A,J-)$ and $F$ preserve them; it follows then that $\eta_T$ is a regular epimorphism for each $T$ in $\D$.
\end{proof}

$ $

\begin{lema}
	Let $\C$ be a small regular $\V$-category, $\L=\textnormal{Lex}(\C,\V)$ and $\D=\textnormal{Reg}(\C,\V)$; denote by $J:\D\hookrightarrow\L$ be the inclusion. Then for any functor $F:\D\to\V$ preserving filtered colimits, the right Kan extension $\textnormal{Ran}_JF:\L\to\V$ preserves filtered colimits too.
\end{lema}
\begin{proof}
	First we prove that $\textnormal{Ran}_JF$ preserves some particular limits. For each $L\in\L$ fix a codensity presentation as in Proposition \ref{codense}:
	\begin{center}
		
		\begin{tikzpicture}[baseline=(current  bounding  box.south), scale=2]
		
		\node (a) at (-0.3,0) {$L$};
		\node (b) at (0.6,0) {$S$};
		\node (c) at (1.6,0) {$M$};
		\node (d) at (2.5,0) {$T$};

		\path[font=\scriptsize]
		
		(a) edge [>->] node [above] {$r$} (b)
		([yshift=1.5pt]b.east) edge [->] node [above] {$u$} ([yshift=1.5pt]c.west)
		([yshift=-1.5pt]b.east) edge [->] node [below] {$v$} ([yshift=-1.5pt]c.west)
		(c) edge [>->] node [above] {$t$} (d);
		\end{tikzpicture}
	\end{center}
	where $r:L\rightarrowtail S$ is a weak reflection of $L$ in $\D$ and a regular monomorphism (Lemma \ref{regfic}), $u,v:S\to M$ is the cokernel pair of $r$, and $t:M\rightarrowtail T$ a weak reflection associated to $M$ (which again we suppose to be a regular monomorphism). Then $t\circ u$ and $t\circ v$ define $L$ as an equalizer of elements from $\D$. By \cite[Theorem~5.29]{Kel82:libro}, $\textnormal{Ran}_JF$ preserves these equalizers. In particular, since $(\text{Ran}_JF)\circ J\cong F$, the object $\textnormal{Ran}_JF(L)$ is defined as the equalizer
	
	\begin{center}
		
		\begin{tikzpicture}[baseline=(current  bounding  box.south), scale=2]
		
		\node (a) at (-0.9,0) {$\text{Ran}_JF(L)$};
		\node (b) at (0.4,0) {$F(S)$};
		\node (c) at (1.9,0) {$F(T)$};
		
		\path[font=\scriptsize]
		
		(a) edge [>->] node [above] {} (b)
		([yshift=1.5pt]b.east) edge [->] node [above] {$F(t\circ u)$} ([yshift=1.5pt]c.west)
		([yshift=-1.5pt]b.east) edge [->] node [below] {$F(t\circ v)$} ([yshift=-1.5pt]c.west);
		\end{tikzpicture}
	\end{center}
	
	To prove that $\text{Ran}_JF$ is finitary, consider the natural transformation $r:1_{\L}\to R$ given by Theorem \ref{wr}, so that each component $r_L:L\to RL$ is a regular monomorphism and a weak reflection of $L$ into $\D$. Form now the cokernel pair $u,v:R\to S$ of $r$; then $S$ is a pushout of finitary functors, hence finitary.  Moreover we can see $r$ as the equalizer
	\begin{center}
		
		\begin{tikzpicture}[baseline=(current  bounding  box.south), scale=2]
		
		\node (a) at (-0.2,0) {$1_{\L}$};
		\node (b) at (0.6,0) {$R$};
		\node (c) at (1.5,0) {$RS$};

		\path[font=\scriptsize]
		
		(a) edge [>->] node [above] {$r$} (b)
		([yshift=1.5pt]b.east) edge [->] node [above] {$rS\circ u$} ([yshift=1.5pt]c.west)
		([yshift=-1.5pt]b.east) edge [->] node [below] {$rS\circ v$} ([yshift=-1.5pt]c.west);
		\end{tikzpicture}
	\end{center}
	with both $R$ and $RS$ finitary. It follows then by the initial observation that $\text{Ran}_JF$ is given by the equalizer
	\begin{center}
		
		\begin{tikzpicture}[baseline=(current  bounding  box.south), scale=2]
		
		\node (a) at (-0.5,0) {$\text{Ran}_JF$};
		\node (b) at (0.6,0) {$FR$};
		\node (c) at (1.9,0) {$FRS$};

		\path[font=\scriptsize]
		
		(a) edge [>->] node [above] {} (b)
		([yshift=1.5pt]b.east) edge [->] node [above] {$F(rS\circ u)$} ([yshift=1.5pt]c.west)
		([yshift=-1.5pt]b.east) edge [->] node [below] {$F(rS\circ v)$} ([yshift=-1.5pt]c.west);
		\end{tikzpicture}
	\end{center}
	that it is a finite limit of finitary functors into the \lfp\ $\V$, and so is itself finitary.
	
\end{proof}

\begin{prop}\label{gencoeq}
	Let $\C$ be a small regular $\V$-category and $$\textnormal{ev}_{\C}:\C\longrightarrow \textnormal{DEF}(\textnormal{Reg}(\C,\V),\V)$$ the evaluation functor. For each $F\in\textnormal{DEF}(\textnormal{Reg}(\C,\V),\V)$ there exist $A,B\in\C$ and maps $f,g:A\to B$ such that $F$ is the coequalizer:
	\begin{center}
		
		\begin{tikzpicture}[baseline=(current  bounding  box.south), scale=2]
		
		\node (b) at (-0.3,0) {$\textnormal{ev}_{\C}(A)$};
		\node (c) at (1.2,0) {$\textnormal{ev}_{\C}(B)$};
		\node (d) at (2.4,0) {$F.$};

		\path[font=\scriptsize]
		
		([yshift=1.5pt]b.east) edge [->] node [above] {$\textnormal{ev}_{\C}(f)$} ([yshift=1.5pt]c.west)
		([yshift=-1.5pt]b.east) edge [->] node [below] {$\textnormal{ev}_{\C}(g)$} ([yshift=-1.5pt]c.west)
		(c) edge [->>] node [above] {} (d);
		\end{tikzpicture}
	\end{center}
	In particular $\textnormal{DEF}(\textnormal{Reg}(\C,\V),\V)$ is a small $\V$-category.\\
	The same holds if we replace $\textnormal{Reg}(\C,\V)$ with $\textnormal{Lex}(\C,\V)$ for any $\V$-category $\C$ with finite weighted limits.
\end{prop}
\begin{proof}
	Denote as before $\D=\textnormal{Reg}(\C,\V)$ and $\L=\textnormal{Lex}(\C,\V)$; let us first prove that the hypotheses of Lemma \ref{regepi} are satisfied. For this, consider a functor $F\in\textnormal{DEF}(\D,\V)$, $L\in\D$, and $x:I\to FL$; then write $L$ as a filtered colimit in $\L$ of finitely presentable objects $L\cong\text{colim}(A_j)$. By the previous Lemma, $G:=\text{Ran}_JF$ preserves filtered colimits, then $GL\cong\text{colim}\ G(A_j)$. Since $I$ is finitely presentable in $\V$, $x$ factors through some colimit map $G(A_j)\to GL\cong FL$; but $G(A_j)\cong[\L,\V](\L(A_j,-),G)$, hence the factorization corresponds to some $\eta:\L(A_j,-)\to G$. Its restriction $\eta J:\L(A_j,J-)\to F$ then satisfies the required property.\\
	Now, thanks to Lemma \ref{regepi}, for each $F\in\textnormal{DEF}(\D,\V)$ there exists a regular epimorphism $\L(B,J-)\twoheadrightarrow F$ with $B\in\L_f$; but $\L_f\simeq \C^{op}$ so the last map corresponds to a regular epimorphism $\eta:\textnormal{ev}_{\C}(B)\twoheadrightarrow F$ for some $B\in\C$. Take then the kernel pair  $\alpha,\beta:F'\to\textnormal{ev}_{\C}(B)$ of $\eta$. Considering again a regular epimorphism $\gamma:\textnormal{ev}_{\C}(A)\twoheadrightarrow F'$ ($A\in\C$) and composing it with $\alpha$ and $\beta$ we obtain $F$ as the coequalizer of objects from $\C$. Finally, the maps defining $F$ as a coequalizer come from $\C$ since the functor $\textnormal{ev}_{\C}:\C\to\textnormal{DEF}(\D,\V)$ is fully faithful.
	
	The same proof applies to the case of $\textnormal{Lex}(\C,\V)$ with the only difference that there is no need to consider the right Kan extension of $F$.
\end{proof}

\begin{obs}\label{goodcoeq}
	The coequalizers defined in the previous proofs are preserved by regular functors since they are given by the composite of a coequalizer of a kernel pair and a regular epimorphism.
\end{obs}

Recall the following result for ordinary categories:

\begin{lema}\cite[Lemma~1.4.9]{MR77:libro}
	Suppose that $F:\B\to\C$ is a conservative, full and regular functor between ordinary regular categories, and $\B$ is exact. If for every object $C\in\C$ there are an object $B\in\B$ and a regular epimorphism $F(B)\twoheadrightarrow C$, then $F$ is an equivalence of categories.
\end{lema}

Then we are ready to prove the following Theorem; the unenriched version appeared originally as \cite[Theorem~5.1]{Mak90:articolo}. Another proof of the unenriched version can be found in \cite[Theorem~2.4.2]{Lur18:articolo}.

\begin{teo}[Makkai's Image Theorem]\label{Makkai}
	For any small exact $\V$-category $\B$; the evaluation map $$\textnormal{ev}_{\B}:\B\longrightarrow\textnormal{DEF}(\textnormal{Reg}(\B,\V),\V)$$ is an equivalence.
\end{teo}
\begin{proof}
	Since $\text{ev}_{\B}$ is fully faithful by Theorem \ref{Barr}, we only need to prove that it is essentially surjective on objects, or equivalently, that the ordinary functor $(\text{ev}_{\B})_0$ is an equivalence. Thanks to Proposition \ref{gencoeq}, for each $F\in\textnormal{DEF}(\textnormal{Reg}(\B,\V),\V)$ there are an object $C\in\B$ and a regular epimorphism $\text{ev}_{\B}(C)\twoheadrightarrow F$; hence we can apply the previous Lemma and conclude.
\end{proof}

\section{Duality for Enriched Exact Categories}\label{equiv}

Let us consider again categories enriched over a base $\V$ which is a symmetric monoidal finitary quasivariety with strong generator $\P\subseteq (\V_0)_{pf}$. 

The following definition comes, slightly modified, from \cite{JK01:articolo}; it will be useful to prove the main result of this section.

\begin{Def}\label{finact}
	Let $\K$ be an ordinary category; a {\em finite action} on $\K$ is an action of the monoidal category $\V_{0f}^{op}$ on $\K$, that is, a functor $$H:\V_{0f}^{op}\times\K\to\K,$$ denoted as $X^A:=H(A,X)$, together with two natural isomorphisms $\alpha$ and $\lambda$ with components $\alpha_{XAB}:X^{A\otimes B}\to (X^A)^B$ and $\lambda_X:A^I\to A$ satisfying the commutativity of diagrams (1.1) and (1.3) in \cite{JK01:articolo}. A finite action $(H,\alpha,\lambda)$ is called {\em closed} if for each $X\in\K$ the functor $X^{(-)}=H(-,X):\V_{0f}^{op}\to\K$ has a left $J^{op}$-adjoint $\bo{K}(-,X):\K\to\V_{0}^{op}$, where $J:\V_{0f}\to\V_{0}$ is the inclusion. In other words, if there is a functor $\bo{K}:\K^{op}\times\K\to\V_{0}$ such that 
	$$ \K(Y,X^A)\cong\V_0(A,\bo{K}(Y,X)) $$
	naturally in $X,Y\in\K$ and $A\in\V_{0f}^{op}$ (equivalently, $\K(Y,X^{(-)}):\V_{0f}^{op}\to\bo{Set}$ preserves finite limits).
\end{Def}

Then we obtain:

\begin{prop}[Appendix of \cite{JK01:articolo}]\label{finpow}
	The forgetful functor induces a biequivalence between:\begin{itemize}
		\item the 2-category of finitely powered $\V$-categories, $\V$-functors which preserve these powers, and $\V$-natural transformations;
		\item the 2-category of ordinary categories with a closed finite action, functors preserving the actions up to coherent natural isomorphism, and natural transformations which are compatible with the action. 
	\end{itemize}
\end{prop}

Even if our setting is different from that of \cite{JK01:articolo}, the same proof applies since $\V_{0f}$ is a strong generator of $\V_0$.

Recall now the definition of (ordinary) regular congruence from \cite{Ben89:articolo}

\begin{Def}
	Let $\C$ be an ordinary regular category. A {\em pullback congruence} on $\C$ is a class $\Sigma$ of maps of $\C$ for which:\begin{itemize}
		
		\item every isomorphism belongs to $\Sigma$;
		\item if $f=h\circ g$ and two of the three maps are in $\Sigma$, so is the third;
		\item $\Sigma$ is pullback stable: for any pullback in $\C$
		\begin{center}
			\begin{tikzpicture}[baseline=(current  bounding  box.south), scale=2]
			
			\node (a) at (0,0.7) {$X'$};
			\node (b) at (0.7,0.7) {$X$};
			\node (c) at (0, 0) {$Y$};
			\node (d) at (0.7, 0) {$Y'$};
			
			\path[font=\scriptsize]
			
			(a) edge [->] node [above] {$g'$} (b)
			(a) edge [->] node [left] {$f'$} (c)
			(b) edge [->] node [right] {$f$} (d)
			(c) edge [->] node [below] {$g$} (d);
			
			\end{tikzpicture}	
		\end{center}
		if $f\in\Sigma$, then $f'\in\Sigma$;
	\end{itemize}
	We call it a {\em regular congruence} if in addition $\Sigma$ is local: for any pullback in $\C$ as above, for which $g$ is a regular epimorphism, if $f'\in\Sigma$ then $f\in\Sigma$.
\end{Def}

In the enriched context we consider the following corresponding notion:

\begin{Def}
	Let $\C$ be a regular $\V$-category. A {\em pullback $\V$-congruence} on $\C$ is a class $\Sigma$ of maps from $\C$ which is a pullback congruence in the ordinary sense and is closed under finite powers (for each $A\in\V_f$, if $h\in\Sigma$ then $h^A\in\Sigma$). Similarly a {\em regular $\V$-congruence} is a pullback $\V$-congruence which is also a regular congruence in the ordinary sense. 
\end{Def}

Then we can now prove:

\begin{prop}\label{frac}
	Let $\C$ be a finitely complete $\V$-category and $\Sigma$ a pullback $\V$-congruence on $\C$. Then the $\V$-category of fractions $\C[\Sigma^{-1}]$ exists in $\V\text{-}\bo{Lex}$. In other words there is a finitely complete $\V$-category $\C[\Sigma^{-1}]$ together with a lex $\V$-functor $$P:\C\to \C[\Sigma^{-1}]$$ such that, for any $\B$ in $\V\text{-}\bo{Lex}$, composition with $P$ induces an equivalence of categories between $\V\text{-}\textnormal{Lex}(\C[\Sigma^{-1}],\B)$ and the full subcategory of 
	$\V\text{-}\textnormal{Lex}(\C,\B)$ consisting of those functors which invert the elements of $\Sigma$.\\
	If $\C$ is regular and $\Sigma$ is a regular $\V$-congruence, then $\C[\Sigma^{-1}]$ is a regular $\V$-category, $P$ is a regular $\V$-functor, and $F_{\Sigma}$ is regular if and only if $F$ is so.
\end{prop}
\begin{proof}
	Consider $\C_0$ as an ordinary regular category; since $\Sigma$ is a pullback congruence the category of fractions $\C_0[\Sigma^{-1}]$ exists and is finitely complete (by Section 1.7 of \cite{Ben89:articolo}). Denote by $P_0:\C_0\to\C_0[\Sigma^{-1}]$ the corresponding lex functor with the usual universal property. We also know that $\C_0[\Sigma^{-1}]$ has the same objects as $\C_0$ and hom-sets given by the filtered colimits
	$$ \C_0[\Sigma^{-1}](X,Y)\cong\underset{X'\to X\in\Sigma}{\text{colim}}\C_0(X',Y) $$
	for each $X$ and $Y$ in $\C_0$ \cite[Proposition~I.2.4]{GZ67:libro}.
	
	By definition $\C$ has finite powers; hence we have an induced finite action $\V_{0f}^{op}\times \C_0\to\C_0$ given by $(A,X)\mapsto X^A$. Since by hypothesis $\Sigma$ is closed under finite powers, this action extends to $\C_0[\Sigma^{-1}]$:
	\begin{center}
		\begin{tikzpicture}[baseline=(current  bounding  box.south), scale=2]
		
		\node (a) at (0,0.8) {$\V_{0f}^{op}\times \C_0$};
		\node (b) at (1.5,0.8) {$\C_0$};
		\node (c) at (0, 0) {$\V_{0f}^{op}\times \C_0[\Sigma^{-1}]$};
		\node (d) at (1.5, 0) {$\C_0[\Sigma^{-1}]$};
		
		\path[font=\scriptsize]
		
		(a) edge [->] node [above] {} (b)
		(a) edge [->] node [left] {$\text{id}\times P_0$} (c)
		(b) edge [->] node [right] {$P_0$} (d)
		(c) edge [->] node [below] {$H$} (d);
		
		\end{tikzpicture}	
	\end{center}
	Then $H:\V_{0f}^{op}\times \C_0[\Sigma^{-1}]\to\C_0[\Sigma^{-1}]$, together with $\alpha_{\Sigma}:=P_0(\alpha_{\C})$ and $\lambda_{\Sigma}:=P_0(\lambda_{\C})$, defines a finite action on $\C_0[\Sigma^{-1}]$ ($\alpha_{\C}$ and $\lambda_{\C}$ being the natural isomorphisms induced by $\C$).
	
	Denote, as in Definition \ref{finact}, $X^A:=H(A,X)$; then for each $X,Y$ in $\C_0[\Sigma^{-1}]$ and $A\in\V_{0f}$ the following hold 
	\begin{equation*}
	\begin{split}
	\C_0[\Sigma^{-1}](X,Y^A)&\cong \underset{X'\to X\in\Sigma}{\text{colim}}\C_0(X',Y^A)\\
	&\cong \underset{X'\to X\in\Sigma}{\text{colim}}\V_0(A,\C(X',Y))\\
	&\cong \V_0(A,\underset{X'\to X\in\Sigma}{\text{colim}}\C(X',Y))
	\end{split}
	\end{equation*}
	where the last holds since $A$ is finitely presentable and the colimit is filtered; this proves that $H$ is a closed finite action on $\C_0[\Sigma^{-1}]$. It then follows from Proposition \ref{finpow} that there exists a $\V$-category with finite powers $\C[\Sigma^{-1}]$ whose underlying ordinary category is $\C_0[\Sigma^{-1}]$, and whose hom-objects are given by 
	$$ \C[\Sigma^{-1}](X,Y)\cong \underset{X'\to X\in\Sigma}{\text{colim}}\C(X',Y); $$
	in addition finite powers in $\C[\Sigma^{-1}]$ are computed as in $\C$.
	
	Next we prove that $\C[\Sigma^{-1}]$ has finite conical limits. We already know that $\C_0[\Sigma^{-1}]$ has ordinary finite limits; in order to show that these are enriched, it suffices to prove that they are preserved by each representable functor $\C[\Sigma^{-1}](X,-)_0: \C_0[\Sigma^{-1}]\to V_0$. And this is true if and only if they are preserved by $\V_0(A,\C[\Sigma^{-1}](X,-)_0):\C_0[\Sigma^{-1}]\to\bo{Set}$, for each $A\in\V_{0f}$. Consider the following commutative (up to isomorphism) diagram: 
	
	\begin{center}
		
		\begin{tikzpicture}[baseline=(current  bounding  box.south), scale=2]
		
		\node (a) at (-0.5,0) {$\C_0$};
		\node (b) at (0.6,0) {$\C_0[\Sigma^{-1}]$};
		\node (c) at (2.2,0) {$\V_0$};
		
		\node (a') at (-0.5,-0.8) {$\C_0$};
		\node (b') at (0.6,-0.8) {$\C_0[\Sigma^{-1}]$};
		\node (c') at (2.2,-0.8) {$\bo{Set}$};
		
		\path[font=\scriptsize]
		
		(a) edge [->] node [above] {$P_0$} (b)
		(b) edge [->] node [above] {$\C[\Sigma^{-1}](X,-)$} (c)
		(a') edge [->] node [above] {$P_0$} (b')
		(b') edge [->] node [above] {$\C[\Sigma^{-1}](X,-)$} (c')
		(a) edge [->] node [left] {$(-)^A$} (a')
		(b) edge [->] node [left] {$(-)^A$} (b')
		(c) edge [->] node [right] {$\V_0(A,-)$} (c');
		\end{tikzpicture}
	\end{center}
	
	Now, the lower composite preserves finite limits by construction; hence, by the universal property of $P_0$, we obtain that $\V_0(A,\C[\Sigma^{-1}](X,-)_0)$ preserves finite limits as desired.
	
	It follows that $\C[\Sigma^{-1}]$ is a finitely complete $\V$-category; moreover, again by Proposition \ref{finpow}, the ordinary functor $P_0$ extends to a $\V$-functor $P:\C\to\C[\Sigma^{-1}]$ preserving finite powers and finite conical limits (because $P_0$ preserves them); hence $P$ preserves all finite weighted limits. 
	
	Finally, let us prove that $P:\C\to\C[\Sigma^{-1}]$ has the required universal property. Let $F:\C\to\B$ be any lex $\V$-functor; if $F$ factors through $P$ then it certainly inverts the arrows in $\Sigma$ since $P$ does. Vice versa, assume that $F$ sends the arrows in $\Sigma$ to isomorphisms; then, by the ordinary universal property, $F_0$ factors through $P_0$ as $F_0=(F_0)_{\Sigma}\circ P_0$, with $(F_0)_{\Sigma}:\C_0[\Sigma^{-1}]\to\B_0$ lex. By construction this functor respects the finite action of $\C_0[\Sigma^{-1}]$ in the sense of Proposition \ref{finpow}; it then follows that $(F_0)_{\Sigma}$ extends to a $\V$-functor $F_{\Sigma}:\C[\Sigma^{-1}]\to\B$ preserving finite powers. As a consequence $F_{\Sigma}$ is a lex $\V$-functor satisfying the required properties (the uniqueness of the factorization follows from the ordinary case).
	
	Assume now that $\Sigma$ is a regular $\V$-congruence, then by Section 2.2 of \cite{Ben89:articolo} the category $\C_0[\Sigma^{-1}]$ and the functor $P$ are regular; we have to show that $\C[\Sigma^{-1}]$ is regular as a $\V$-category and $P$ is a regular $\V$-functor. We already know that $\C[\Sigma^{-1}]$ is finitely complete and $\C_0[\Sigma^{-1}]$ is regular in the ordinary sense; in addition $P$ preserves regular epimorphisms since $P_0$ does. It only remains to prove that regular epimorphisms in $\C[\Sigma^{-1}]$ are stable under powers with finite projective objects. Let $e$ be a regular epimorphism in $\C[\Sigma^{-1}]$ and $A\in(\V_0)_{pf}$; by \cite[Theorem~2.2.2.(7)]{Ben89:articolo}, $e\cong P(h)$ for a regular epimorphism $h$ in $\C$; then $e^A\cong P(h)^A\cong P(h^A)$ is a regular epimorphism because $h^A$ is and $P$ preserves them.
	The fact that $F_{\Sigma}$ is regular if and only if $F$ is follows directly from the previous and the ordinary cases.
\end{proof}

The following gives another characterization of definable $\V$-categories:

\begin{prop}\label{defreg}
	Let $\D$ be a $\V$-category; the following are equivalent:\begin{enumerate}
		\item $\D$ is a definable $\V$-category;
		\item there exists a regular $\V$-category $\C$ such that $\D\simeq\textnormal{Reg}(\C,\V)$.
	\end{enumerate}
\end{prop}
\begin{proof}
	$(2)\Rightarrow(1)$ follows from Remark \ref{regdef}. In order to prove $(1)\Rightarrow(2)$, let $\D=\M$-inj be a finite injectivity class in a \lfp\ $\V$-category $\L$. Write $\L$ as $\text{Lex}(\A,\V)$ where $\A=\L_f^{op}$; then $\M$ can be identified with a collection of morphisms in $\A$, and $\D$ coincides with the full subcategory of $\text{Lex}(\A,\V)$ given by those $\V$-functors that send each $h\in\M$ to a regular epimorphism. Consider now $\C'=\A_{\text{reg/lex}}$ to be the free regular $\V$-category over $\A$ (this exists by Remark \ref{reg/lex}); it follows that $\text{Lex}(\A,\V)\simeq\text{Reg}(\C',\V)$. Under this equivalence $\M$ corresponds to a small set of arrows in $\C'$, and $\D$ to the full subcategory of $\text{Reg}(\C',\V)$ given by those regular $\V$-functors that send each $h\in\M$ to a regular epimorphism. For each $h\in\M$ take its image factorization $h=m_h\circ e_h$ in $\C'$, where $e_h$ is a regular epimorphism and $m_h$ a monomorphism; then a regular $\V$-functor $F:\C'\to\V$ sends $h$ to a regular epimorphism in $\V$ if and only if it sends $m_h$ to an isomorphism. Thus $\D$ corresponds to the full subcategory of $\text{Reg}(\C',\V)$ given by those regular $\V$-functors that invert the maps in the set $\N=\{\ m_h \ |\ h\in\M\ \}$. Let now $\Sigma$ be the saturation of $\N$ with respect to $\D$:
	$$ \Sigma:=\{ f\in\C'\ |\  F(f)\text{ is invertible for each } F\in\D  \}\supseteq\N, $$
	where we are seeing $\D$ in $\text{Reg}(\C',\V)$ as above. Then a regular $\V$-functor $F:\C'\to\V$ is $\M$-injective if and only if it inverts each $f\in\Sigma$. It's easy to check that $\Sigma$ is a regular $\V$-congruence in $\C'$; hence by the previous Proposition, $\C:=\C'[\Sigma^{-1}]$ exists as a regular $\V$-category and by construction $\D\simeq\text{Reg}(\C,\V)$. 
\end{proof}

Let $\D$ be a definable $\V$-category; then by the previous Proposition there is a regular $\V$-category $\C$ for which $\D\simeq\text{Reg}(\C,\V)$. As a consequence $\textnormal{DEF}(\D,\V)\simeq\textnormal{DEF}(\text{Reg}(\C,\V),\V)$ is a small $\V$-category by Proposition \ref{gencoeq}; since $\V$ is regular as a $\V$-category, $\textnormal{DEF}(\D,\V)$ is a regular $\V$-category too (being closed in $[\D,\V]$ under finite limits and coequalizers of kernel pairs). Moreover, given any regular $\V$-category $\C$, the regular $\V$-functors from $\C$ to $\V$ form a definable subcategory $\text{Reg}(\C,\V)$ of $\text{Lex}(\C,\V)$ (as shown in Remark \ref{regdef}). As a consequence we obtain an adjunction

\begin{center}\label{adj}
	
	\begin{tikzpicture}[baseline=(current  bounding  box.south), scale=2]

	\node (f) at (0,0.4) {$\V\text{-}\bo{DEF}$};
	\node (g) at (1.8,0.4) {$\V\text{-}\bo{Reg}^{op}$};
	\node (h) at (0.9,0.4) {$\perp$};
	
	\path[font=\scriptsize]

	([yshift=2.3pt]f.east) edge [->] node [above] {$\textnormal{DEF}(-,\V)$} ([yshift=2.3pt]g.west)
	([yshift=-2.3pt]f.east) edge [<-] node [below] {$\textnormal{Reg}(-,\V)$} ([yshift=-2.3pt]g.west);
	\end{tikzpicture}
	
\end{center}
of 2-categories. Indeed for each regular $\C$ and each definable $\V$-category $\D$ the following holds
$$\V\text{-}\bo{DEF}(\D,\textnormal{Reg}(\C,\V))\cong \V\text{-}\bo{Reg}(\C,\textnormal{DEF}(\D,\V))$$
since each is isomorphic to the category of $\V$-functors $\D\otimes\C\to\V$ which are definable in the first variable and regular in the second.

The counit and unit of this adjunction are given by the evaluation functors: $$\text{ev}_{\C}:\C\to\text{DEF}(\text{Reg}(\C,\V),\V)$$ for a regular $\C$, and $$\text{ev}_{\D}:\D\to\text{Reg}(\text{DEF}(\D,\V),\V)$$ for a definable $\D$. We already saw in Theorems \ref{Barr} and \ref{Makkai} that the counit $\text{ev}_{\C}$ is a fully faithful functor, and an equivalence if $\C$ is moreover exact. 

Now assume that our base for enrichment $\V$ is a symmetric monoidal finitary variety; then $\V_0$ is an exact category and $\V$ is exact as a $\V$-category. As a consequence $\text{DEF}(\D,\V)$ is a small exact $\V$-category for each definable $\D$; hence the 2-adjunction between $\V\text{-}\bo{DEF}$ and $\V\text{-}\bo{Reg}^{op}$ restricts to

\begin{center}
	
	\begin{tikzpicture}[baseline=(current  bounding  box.south), scale=2]

	\node (f) at (0,0.4) {$\V\text{-}\bo{DEF}$};
	\node (g) at (1.8,0.4) {$\V\text{-}\bo{Ex}^{op}$};
	\node (h) at (0.9,0.4) {$\perp$};
	
	\path[font=\scriptsize]

	([yshift=2.3pt]f.east) edge [->] node [above] {$\textnormal{DEF}(-,\V)$} ([yshift=2.3pt]g.west)
	([yshift=-2.3pt]f.east) edge [<-] node [below] {$\textnormal{Reg}(-,\V)$} ([yshift=-2.3pt]g.west);
	\end{tikzpicture}
	
\end{center}
where $\V\text{-}\bo{Ex}$ is the 2-category of all small exact $\V$-categories, regular $\V$-functors, and $\V$-natural transformations. Since by Theorem \ref{Makkai} the counit of this adjunction is an equivalence, it follows that $\textnormal{Reg}(-,\V)$ is bi-fully faithful  (an equivalence on the categories of homomorphisms). Moreover by Proposition \ref{defreg}, each definable $\D$ is equivalent to $\textnormal{Reg}(\C,\V)$ for some regular $\V$-category $\C$; taking $\B$ to be $\C_{\text{ex/reg}}$, the free exact $\V$-category on $\C$ as a regular $\V$-category (which exists by Remark \ref{reexreg}), we then obtain $\D\simeq\textnormal{Reg}(\B,\V)$. This means that $\textnormal{Reg}(-,\V):\V\text{-}\bo{Ex}^{op}\to \V\text{-}\bo{DEF}$ is also essentially surjective. As a consequence $\textnormal{Reg}(-,\V)$ is a biequivalence with inverse $\textnormal{DEF}(-,\V)$. Thence we have proven:

\begin{teo}\label{main}
	Let $\V$ be a symmetric monoidal finitary variety. Then the 2-adjunction
	\begin{center}
		
		\begin{tikzpicture}[baseline=(current  bounding  box.south), scale=2]

		\node (f) at (0,0.4) {$\textnormal{DEF}(-,\V):\V\text{-}\bo{DEF}$};
		\node (g) at (2.2,0.4) {$\V\text{-}\bo{Ex}^{op}:\textnormal{Reg}(-,\V)$};
		
		\path[font=\scriptsize]

		([yshift=1.3pt]f.east) edge [->] node [above] {} ([yshift=1.3pt]g.west)
		([yshift=-1.3pt]f.east) edge [<-] node [below] {} ([yshift=-1.3pt]g.west);
		\end{tikzpicture}
		
	\end{center}
	is a biequivalence.
\end{teo}

This duality was first shown for the additive case in \cite[Theorem~2.3]{PR10:articolo}, in this context it becomes a biequivalence between the 2-category of additive definable categories and the opposite of the 2-category of abelian categories. The ordinary version appeared more recently in \cite[Theorem~3.2.5]{KR18:articolo}. As we anticipated in the introduction, the proof appearing there is incomplete; more precisely the proof of \cite[Proposition~3.2.2]{KR18:articolo} contains the following unjustified isomorphism which affects the proof of the duality:
\begin{center}
	`` $[\L,\bo{Set}](\textnormal{lim}(y'K_s),X)\simeq \textnormal{colim}[\L,\bo{Set}](y'K_s,X)$ ''.
\end{center}
where $y'$ is the Yoneda embedding. Moreover, the sort of epimorphism guaranteed in Corollary 3.2.3 does not seem to match that used in Theorem 3.2.4. Our Theorem \ref{main} provides a solution for this.

An immediate consequence is:

\begin{cor}\label{evaldef}
	Let $\D$ be a definable $\V$-category, where $\V$ is a symmetric monoidal finitary variety; then the evaluation functor 
	$$ \textnormal{ev}_{\D}:\D\longrightarrow\textnormal{Reg}(\textnormal{DEF}(\D,\V),\V) $$
	is an equivalence of $\V$-categories.
\end{cor}

\begin{obs}
	It is worth pointing out that the previous Corollary still holds even if $\V$ is just a symmetric monoidal finitary quasivariety; a proof for this can be found in \cite[Proposition~4.3.5]{Ten19:articolo} (the notion of definable $\V$-category appearing there is equivalent to the one we use thanks to Proposition \ref{defreg}; note moreover that the notion of symmetric monoidal finitary quasivariety is slightly different, but that doesn't affect the proofs).
\end{obs}

\section{Free Exact $\V$-Categories}\label{free}

Consider again $\V$ to be a symmetric monoidal finitary variety as in the last part of the previous section. We are going to use Theorem \ref{main} to find the free exact $\V$-categories associated to finitely complete and regular ones.

In the ordinary context, exact completions over regular categories were first considered in \cite{Law73:articolo}; while regular and exact completions over finitely complete categories have been dealt with in \cite{CM82:articolo} and \cite{CV98:articolo}. A different, but equivalent, description of them has been given in \cite{Hu96:articolo} and \cite{HT96:articolo}, where exact completions are built as certain categories of functors preserving determined limits and colimits. Yet another description of the free exact category on a regular one was given in \cite{Lac99:articolo}.

\begin{prop}\label{freeexfc}
	Let $\C$ be a finitely complete $\V$-category and define $\L=\textnormal{Lex}(\C,\V)$. Then for each small exact $\V$-category $\B$, precomposition with $\textnormal{ev}:\C\to\textnormal{DEF}(\L,\V)$ induces an equivalence:
	$$ \textnormal{Reg}(\textnormal{DEF}(\L,\V),\B)\simeq\textnormal{Lex}(\C,\B).$$
	In other words $\textnormal{DEF}(\L,\V)$ is the free exact $\V$-category over $\C$ as a finitely complete $\V$-category.
\end{prop}
\begin{proof}
	We are to show that the composite
	$$\textnormal{Reg}(\textnormal{DEF}(\L,\V),\B)\hookrightarrow \textnormal{Lex}(\textnormal{DEF}(\L,\V),\B)\stackrel{-\circ\text{ev}}{\longrightarrow} \textnormal{Lex}(\C,\B)$$ 
	is an equivalence for all small exact $\V$-categories $\B$. First observe that by Corollary \ref{evaldef} this is the case for $\B=\V$. As a consequence precomposition with $\textnormal{ev}$ induces the desired equivalence also for all $\B$ of the form $[\A,\V]$.
	Let now $\B$ be any small exact category; by Theorem \ref{functembedd} we can assume that $\B$ is a full subcategory of some $[\A,\V]$ with the inclusion $H:\B\to[\A,\V]$ a regular functor. Since the equivalence holds for $[\A,\V]$ we can consider the commutative square
	\begin{center}
		\begin{tikzpicture}[baseline=(current  bounding  box.south), scale=2]
		
		\node (a) at (0,0.8) {$\textnormal{Reg}(\textnormal{DEF}(\L,\V),\B)$};
		\node (b) at (2.4,0.8) {$\textnormal{Lex}(\C,\B)$};
		\node (c) at (0, 0) {$\textnormal{Reg}(\textnormal{DEF}(\L,\V),[\A,\V])$};
		\node (d) at (2.4, 0) {$\textnormal{Lex}(\C,[\A,\V])$};

		\path[font=\scriptsize]
		
		(a) edge [->] node [above] {$-\circ\textnormal{ev}$} (b)
		(b) edge [->] node [right] {$H\circ-$} (d)
		(a) edge [->] node [left] {$H\circ-$} (c)
		(c) edge [] node [above] {$\simeq$} (d)
		(c) edge [->] node [below] {$-\circ\textnormal{ev}$} (d);
		\end{tikzpicture}	
	\end{center}
	in which the bottom arrow is an equivalence and the vertical ones fully faithful. Thus the upper horizontal is fully faithful, and it is enough to prove that, given $F\in\textnormal{Lex}(\C,\B)$, the induced extension of $HF$ to $\textnormal{DEF}(\L,\V)$ takes values in $\B$. Let $G$ be the mentioned extension of $HF$, then we can consider the diagram
	\begin{center}
		\begin{tikzpicture}[baseline=(current  bounding  box.south), scale=2]
		
		\node (a) at (0,0.8) {$\C$};
		\node (b) at (1.2,0.8) {$\B$};
		\node (c) at (0, 0) {$\textnormal{DEF}(\L,\V)$};
		\node (d) at (1.2, 0) {$[\A,\V]$};

		\path[font=\scriptsize]
		
		(a) edge [->] node [above] {$F$} (b)
		(b) edge [->] node [right] {$H$} (d)
		(a) edge [->] node [left] {$\textnormal{ev}$} (c)
		(c) edge [->] node [below] {$G$} (d);
		\end{tikzpicture}	
	\end{center}
	The commutativity of this square (up to isomorphism) says that $G$ restricted to the evaluation functors $\textnormal{ev}(C)$, for $C\in\C$, takes values in $\B$. Given any other $M\in \textnormal{DEF}(\L,\V)$, by Proposition \ref{gencoeq}, we can write it as a coequalizer:
	\begin{center}
		
		\begin{tikzpicture}[baseline=(current  bounding  box.south), scale=2]
		
		\node (a) at (-0.3,0) {$\textnormal{ev}(C)$};
		\node (b) at (0.6,0) {$N$};
		\node (c) at (1.6,0) {$\textnormal{ev}(D)$};
		\node (d) at (2.5,0) {$M$};

		\path[font=\scriptsize]
		
		(a) edge [->>] node [above] {$\gamma$} (b)
		([yshift=1.5pt]b.east) edge [->] node [above] {$\alpha$} ([yshift=1.5pt]c.west)
		([yshift=-1.5pt]b.east) edge [->] node [below] {$\beta$} ([yshift=-1.5pt]c.west)
		(c) edge [->>] node [above] {$\eta$} (d);
		\end{tikzpicture}
	\end{center}
	where $(\alpha,\beta)$ is the kernel pair of $\eta$ and, since ev is fully faithful, $\alpha\circ\gamma=\text{ev}(u)$ and $\beta\circ\gamma=\text{ev}(v)$ for some $u,v:C\to D$ in $\C$. Since $G$ preserves finite limits and coequalizers of kernel pairs, the image of the previous diagram under $G$ is 
	\begin{center}
		
		\begin{tikzpicture}[baseline=(current  bounding  box.south), scale=2]
		
		\node (a) at (-0.3,0) {$FC$};
		\node (b) at (0.6,0) {$GN$};
		\node (c) at (1.6,0) {$FD$};
		\node (d) at (2.5,0) {$GM$};
		
		\path[font=\scriptsize]
		
		(a) edge [->>] node [above] {$G\gamma$} (b)
		([yshift=1.5pt]b.east) edge [->] node [above] {$G\alpha$} ([yshift=1.5pt]c.west)
		([yshift=-1.5pt]b.east) edge [->] node [below] {$G\beta$} ([yshift=-1.5pt]c.west)
		(c) edge [->>] node [above] {$G\eta$} (d);
		\end{tikzpicture}
	\end{center}
	where $G\alpha,G\beta:GN\to FD$ form the kernel pair of $G\eta$. Since $[\A,\V]$ is regular, $GN$ and $(G\alpha,G\beta)$ are given by the image factorization of $(Fu,Fv):FC\to FD\times FD$ and hence $GN$ is actually in $\B$ and $G\alpha$ and $G\beta$ exist as arrows of $\B$. Moreover, being a kernel pair (in $[\A,\V]$) the pair $(G\alpha,G\beta)$ is an equivalence relation. But $\B$ is exact and therefore all equivalence relations are effective; this means that $G\alpha$ and $G\beta$ have a coequalizer in $\B$ which hence coincides with $GM$. As a consequence $G$ takes values in $\B$ as claimed.\\ 
\end{proof}

This says that the left biadjoint to the forgetful functor $U_{\text{ex/lex}}:\V\text{-}\bo{Ex}\to\V\text{-}\bo{Lex}$ is given by the composite
\begin{center}
	
	\begin{tikzpicture}[baseline=(current  bounding  box.south), scale=2]
	
	\node (a) at (0,0) {$\V\text{-}\bo{Lex}$};
	\node (b) at (1.5,0) {$\V\text{-}\bo{LFP}^{op}$};
	\node (c) at (3,0) {$\V\text{-}\bo{DEF}^{op}$};
	\node (d) at (4.5,0) {$\V\text{-}\bo{Ex}$};

	\path[font=\scriptsize]
	
	(a) edge [->] node [above] {$\text{Lex}(-,\V)$} (b)
	(b) edge [->] node [above] {$U_{\text{lfp/def}}^{op}$} (c)
	(c) edge [->] node [above] {$\text{DEF}(-,\V)$} (d);
	\end{tikzpicture}
\end{center}
where $U_{\text{lfp/def}}:\V\text{-}\bo{LFP}\to \V\text{-}\bo{DEF}$ is the forgetful functor. Since the first and the last are actually biequivalences, it follows that $U_{\text{lfp/def}}$ has a left biadjoint too, which is given by

\begin{center}
	
	\begin{tikzpicture}[baseline=(current  bounding  box.south), scale=2]
	
	\node (a) at (0,0) {$\V\text{-}\bo{DEF}$};
	\node (b) at (1.5,0) {$\V\text{-}\bo{Ex}^{op}$};
	\node (c) at (3,0) {$\V\text{-}\bo{Lex}^{op}$};
	\node (d) at (4.5,0) {$\V\text{-}\bo{LFP}.$};

	\path[font=\scriptsize]
	
	(a) edge [->] node [above] {$\text{DEF}(-,\V)$} (b)
	(b) edge [->] node [above] {$U_{\text{ex/lex}}^{op}$} (c)
	(c) edge [->] node [above] {$\text{Lex}(-,\V)$} (d);
	\end{tikzpicture}
\end{center}

The next proposition gives an explicit description of the free exact $\V$-category on a regular one:

\begin{prop}
	Let $\C$ be a small regular $\V$-category and $\R=\textnormal{Reg}(\C,\V)$. Then for each small exact $\V$-category $\B$, precomposition with $\textnormal{ev}:\C\to\textnormal{DEF}(\R,\V)$ induces an equivalence:
	$$ \textnormal{Reg}(\textnormal{DEF}(\R,\V),\B)\simeq\textnormal{Reg}(\C,\B). $$
	In other words $\textnormal{DEF}(\R,\V)$ is the free exact $\V$-category over $\C$ as a regular $\V$-category.
\end{prop}
\begin{proof}
	Note that $\R$ is a definable subcategory of $\text{Lex}(\C,\V)$, hence by Theorem \ref{main} the equivalence $$ \textnormal{Reg}(\C,\V)=\R\simeq \textnormal{Reg}(\textnormal{DEF}(\R,\V),\V) $$
	holds and is induced by the evaluation map. Arguing as in the preceding proof we obtain the equivalence for any small exact $\B$ in place of $\V$.
\end{proof}

As before, this says that the left biadjoint to the forgetful functor $U_{\text{ex/reg}}:\V\text{-}\bo{Ex}\to\V\text{-}\bo{Reg}$ is given by the composite
\begin{center}
	
	\begin{tikzpicture}[baseline=(current  bounding  box.south), scale=2]
	
	\node (a) at (0,0) {$\V\text{-}\bo{Reg}$};
	\node (c) at (1.5,0) {$\V\text{-}\bo{DEF}^{op}$};
	\node (d) at (3,0) {$\V\text{-}\bo{Ex}.$};

	\path[font=\scriptsize]
	
	(a) edge [->] node [above] {$\text{Reg}(-,\V)$} (c)
	(c) edge [->] node [above] {$\text{DEF}(-,\V)$} (d);
	\end{tikzpicture}
\end{center}
$ $

\section{The Infinitary Case}

As often happens, the results we have proven extend to the infinitary case with no particular effort, simply replacing ``finite'' by ``less than $\alpha$'' everywhere (where $\alpha$ is an infinite regular cardinal). In this section we explain in detail how this generalization works.

Let us fix then an infinite regular cardinal $\alpha$; our base for enrichment will now be an $\alpha$-quasivariety:

\begin{Def}
	Let $\V=(\V_0,\otimes,I)$ be a symmetric monoidal closed category. We say that $\V$ is a {\em symmetric monoidal $\alpha$-quasivariety }if:\begin{enumerate}
		
		\item $\V_0$ is an $\alpha$-quasivariety: there is a strong generator $\P\subseteq(\V_{0})_{p\alpha}$ made of $\alpha$-presentable projective objects;
		\item $I\in(\V_{0})_{\alpha}$;
		\item if $P,Q\in\P$ then $P\otimes Q\in(\V_{0})_{p\alpha}$.
	\end{enumerate}
	We call $\V$ a {\em symmetric monoidal $\alpha$-variety} if $\V_0$ is moreover exact.
\end{Def}

\noindent Here we are denoting by $(\V_{0})_{p\alpha}$ the full subcategory of $\alpha$-presentable projective objects of $\V_0$. In particular, if $\gamma$ is a regular cardinal greater than $\alpha$, each symmetric monoidal $\alpha$-quasivariety is also a symmetric monoidal $\gamma$-quasivariety.

Note that any symmetric monoidal $\alpha$-quasivariety $\V$ is locally $\alpha$-presentable as a closed category; hence the Gabriel-Ulmer duality between $\gamma$-complete $\V$-categories and locally $\gamma$-presentable ones still holds for each regular cardinal $\gamma\geq\alpha$, as explained in \cite[Section~7.4]{Kel82:articolo}. Moreover, $\alpha$-filtered colimits commute with $\alpha$-small weighted limits in each locally $\alpha$-presentable $\V$-category.

\begin{Def}
	A $\V$-category $\C$ is said to be {\em $\alpha$-regular} if it has all $\alpha$-small weighted limits, coequalizers of kernel pairs, and is such that regular epimorphisms are stable under pullback and closed under powers by elements of $\P$ and under $\alpha$-small products. A $\V$-functor $F:\C\to\D$ between $\alpha$-regular $\V$-categories is called {\em $\alpha$-regular} if it preserves $\alpha$-small weighted limits and regular epimorphisms; we denote by $\alpha\textnormal{-Reg}(\C,\D)$ the $\V$-category of regular functors from $\C$ to $\D$.\\
	A $\V$-category $\B$ is called $\alpha$-exact if it is $\alpha$-regular and in addition the ordinary category $\B_0$ is exact in the usual sense. Denote by $(\V,\alpha)\text{-}\bo{Ex}$ the 2-category of all small $\alpha$-exact $\V$-categories, $\alpha$-regular $\V$-functors, and $\V$-natural transformations.
\end{Def}

We need regular epimorphisms to be stable under $\alpha$-small products to recover an infinitary version of Proposition \ref{coreg}. In fact, we want each pushout diagram with one specified arrow a regular monomorphism to be an $\alpha$-filtered colimit of representable diagrams of the same kind (in the $\V$-category of $\alpha$-continuous $\V$-functors); this is done like in the finitary case, but to make the colimit $\alpha$-filtered we require the additional condition on $\alpha$-small products. The remaining arguments used in Section \ref{BarrEmb} generalize easily to this context leading to an infinitary version of Barr's Embedding Theorem (the comments made before Theorem \ref{Barr} still apply):

\begin{teo}
	Let $\V$ be a symmetric monoidal $\alpha$-quasivariety.
	For any small $\alpha$-regular $\V$-category $\C$ the evaluation functor $$\textnormal{ev}_{\C}:\C\longrightarrow[\alpha\textnormal{-Reg}(\C,\V),\V]$$ is fully faithful and $\alpha$-regular.
\end{teo}
 
Similarly, from Section \ref{MakIm} we infer an infinitary version of Makkai's Image Theorem:

\begin{teo}
	Let $\V$ be a symmetric monoidal $\alpha$-quasivariety. For any small $\alpha$-exact $\V$-category $\B$ the evaluation map $$\textnormal{ev}_{\B}:\B\longrightarrow\alpha\textnormal{-Def}(\alpha\textnormal{-Reg}(\B,\V),\V)$$ is an equivalence, where $\alpha\textnormal{-Def}(\alpha\textnormal{-Reg}(\B,\V),\V)$ is the full subcategory of $[\alpha\textnormal{-Reg}(\B,\V),\V]$ given by those functors preserving products, projective powers, and $\alpha$-filtered colimits.
\end{teo}

Finally the corresponding notions of $\alpha$-injectivity class and $\alpha$-definable $\V$-category are given as follows:

\begin{Def}
	Let $\M$-inj be an injectivity class of a locally $\alpha$-presentable $\V$-category; if the arrows in $\M$ have $\alpha$-presentable domain and codomain, we call $\M$-inj an {\em enriched $\alpha$-injectivity class}. A $\V$-category $\D$ is then called {\em $\alpha$-definable} if it is an $\alpha$-injectivity class in some locally $\alpha$ presentable $\V$-category. A morphism between definable $\V$-categories is a $\V$-functor that preserves products, projective powers and $\alpha$-filtered colimits. Denote by $(\V,\alpha)\text{-}\bo{DEF}$ the 2-category of $\alpha$-definable $\V$-categories, morphisms between them, and $\V$-natural transformations.
\end{Def}

Then the results of Sections \ref{def} and \ref{equiv} have a suitable extension to this context. The only thing we need to point out is the corresponding infinitary notion of regular $\V$-congruence, needed to obtain the analogue of Proposition \ref{frac}. An {\em $\alpha$-regular $\V$-congruence} on an $\alpha$-regular $\V$-category $\C$ is an ordinary regular congruence $\Sigma$ which is closed under $\alpha$-presentable powers and $\alpha$-small products. The last assumption ensures that colimits indexed on $\Sigma^{op}$ are $\alpha$-filtered; this way both the ordinary construction (from \cite{Ben89:articolo}) and Proposition \ref{frac} extend with no particular changes to an infinitary version. In the end we obtain:

\begin{teo}
	Let $\V$ be a symmetric monoidal $\alpha$-variety. Then the 2-adjunction
	\begin{center}
		
		\begin{tikzpicture}[baseline=(current  bounding  box.south), scale=2]

		\node (f) at (0,0.4) {$(\V,\alpha)\text{-}\bo{DEF}$};
		\node (g) at (2.4,0.4) {$(\V,\alpha)\text{-}\bo{Ex}^{op}$};
		\node (h) at (1.2,0.4) {$\perp$};
		
		\path[font=\scriptsize]

		([yshift=2.3pt]f.east) edge [->] node [above] {$\alpha\textnormal{-Def}(-,\V)$} ([yshift=2.3pt]g.west)
		([yshift=-2.3pt]f.east) edge [<-] node [below] {$\alpha\textnormal{-Reg}(-,\V)$} ([yshift=-2.3pt]g.west);
		\end{tikzpicture}
		
	\end{center}
	is a biequivalence of 2-categories.
\end{teo}

\end{document}